\documentclass[final, 11pt, oneside]{amsart}

\usepackage{ifdraft}

\usepackage[backref=true,style=alphabetic, abbreviate=false]{biblatex}
\ifdraft{\usepackage[notref, notcite]{showkeys}}{}
\usepackage[utf8]{inputenc}
\usepackage{csquotes}
\usepackage[symbol, perpage]{footmisc}
\usepackage[letterpaper, margin=2.7cm]{geometry}
\usepackage{xargs}
\usepackage{amsmath, amsthm, mathrsfs, amssymb, amsfonts}
\usepackage[all,cmtip]{xy}
\usepackage{graphicx}
\usepackage[usenames, dvipsnames, svgnames, table]{xcolor}
\usepackage{pdfsync}
\usepackage{enumitem}
\usepackage{tikz}
\usetikzlibrary{cd}

\usepackage[english]{babel}

\usepackage{imakeidx}
\usepackage{appendix}
\usepackage{verbatim}
\usepackage[colorlinks=true, citecolor=green!50!black, linkcolor=red!50!black, final=true]{hyperref}
\usepackage{cleveref}

\usepackage[colorinlistoftodos,prependcaption,obeyFinal,textsize=tiny]{todonotes} 

\theoremstyle{plain}
\newtheorem{thm}{Theorem}[section] 
\newtheorem{theorem}[thm]{Theorem}

\newtheorem{lem}[thm]{Lemma}
\newtheorem{lemma}[thm]{Lemma}
\newtheorem{cor}[thm]{Corollary}

\newtheorem{prop}[thm]{Proposition}

\newtheorem{proposition}[thm]{Proposition}

\newtheorem*{claim*}{Claim} 
\newtheorem*{thm*}{Theorem}

\theoremstyle{definition}

\theoremstyle{remark}

\newtheorem{rmk}[thm]{Remark}
\newtheorem{remark}[thm]{Remark}

\newtheorem{example}[thm]{Example}

\newtheorem{question}[thm]{Question}

\newcommand{\Gl}{\operatorname{GL}}
\newcommand{\Mat}{\operatorname{M}}

\newcommand{\Mor}{\operatorname{Mor}}

\newcommand{\tensor}{\otimes}
\newcommand{\into}{\hookrightarrow}

\newcommand{\Gm}{{\mathbb{G}_m}}
\newcommand{\Ga}{{\mathbb{G}_a}}
 \newcommand{\Rank}{\operatorname{Rank}}

\renewcommand{\P}{\operatorname{P}}

\newcommand{\sh}[1]{\mathcal{#1}}

\newcommand{\weq}{\simeq}

\newcommand{\GL}{\Gl}

\newcommand{\SO}{\operatorname{SO}}

\newcommand{\Hoh}{\mathrm{H}}

\newcommand{\id}{\mathrm{id}}
\newcommand{\isom}{\cong}
\newcommand{\iso}{\isom}

\newcommand{\mto}[1]{\stackrel{#1}{\longrightarrow}}
\newcommand{\isomto}{\mto{\isom}}

\newcommand{\sm}{-}

\newcommand{\C}{\mathbb{C}}

\newcommand{\Q}{\mathbb{Q}}

\newcommand{\Z}{\mathbb{Z}}

\newcommand{\ZZ}{\Z}
\newcommand{\PP}{\mathbb{P}}
\newcommand{\CC}{\C}
\newcommand{\QQ}{\Q}

\newcommand{\bd}{\partial}

\newcommand{\Aut}{\operatorname{Aut}}

\newcommand{\Aone}{{{\mathbb A}^1}}

\newcommand{\Ker}{\operatorname{Ker}}
\newcommand{\nullity}{\operatorname{null}}

\newcommand{\Gr}{\operatorname{Gr}}

\newcommand{\Kdim}{\operatorname{Kdim}}

\newcommand{\bbP}{\mathbb P}
\renewcommand{\P}{\mathbb P}
\newcommand{\Spec}{\operatorname{Spec}}
\newcommand{\Span}{\operatorname{Span}}

\newcommand{\A}{\mathbb A}

\renewcommand{\phi}{\varphi}

\newcommand{\on}[1]{\operatorname{#1}}

\newcommand{\PGL}{\operatorname{PGL}}

\newcommand{\catC}{{\mathscr{C}}}
\newcommand{\gen}{\operatorname{gen}}

\newcommand{\Floor}[1]{\left\lfloor{#1}\right\rfloor}

\renewcommand{\le}{\leqslant}
\renewcommand{\ge}{\geqslant}

\renewcommand{\theta}{\vartheta}

\newcommand{\benw}[2][]{\todo[linecolor=Green,backgroundcolor=Green!25,bordercolor=Green,#1]{#2---Ben W.}}

\numberwithin{equation}{section}

\author{Uriya A. First}
\address
{Department of Mathematics \\
Faculty of Natural Sciences, University of Haifa, 
Mount Carmel, Haifa, 3498838 \\
Israel}
\email{ufirst@univ.haifa.ac.il}

\author{Zinovy Reichstein}
\address{Department of Mathematics\\ University of British Columbia\\ Vancouver BC V6T 1Z2\\ Canada}
\email{reichst@math.ubc.ca} \email{tbjw@math.ubc.ca}

\author{Ben Williams}

\thanks {The second and third authors have been partially 
supported by 
Discovery Grants from the Natural Sciences and Engineering Research Council of Canada.}

\makeatletter
\@namedef{subjclassname@2020}{%
  \textup{2020} Mathematics Subject Classification}
\makeatother



\begin{document}
\frenchspacing

\title[The number of generators of an algebra]{On the number of generators of an algebra over a commutative ring}

\begin{abstract} A theorem of O.~Forster says that if $R$ is a noetherian ring of Krull dimension $d$, then every projective $R$-module of rank $n$ can be generated by $d+n$ elements. S.~Chase and R.~Swan subsequently showed that this bound is sharp: there exist examples that cannot be generated 
by fewer than $d+n$ elements. We view rank-$n$ projective $R$-modules as \emph{$R$-forms} of the non-unital $R$-algebra $R^n$ where the product of any two elements  
is $0$. The first two authors generalized Forster's theorem to forms of other algebras (not necessarily commutative, associative or unital); 
A. Shukla and the third author then showed that this generalized Forster bound is optimal for \'etale algebras.

In this paper, we prove new upper and lower bounds on the number of generators of an $R$-form of a $k$-algebra, 
where $k$ is an infinite field and $R$ is a finitely generated $k$-ring of Krull dimension $d$. In particular, we show that, 
contrary to expectations, for most types of algebras, the generalized Forster bound is far from optimal. Our results 
are particularly detailed in the case of Azumaya algebras. Our proofs are based on reinterpreting the problem as a question about
approximating the classifying stack $BG$, where $G$ is the automorphism group of the algebra in question, by algebraic spaces of 
a certain type.
\end{abstract}

\maketitle


\section{Introduction} 

Throughout this paper $k$ will denote a base field, $A$ a finite-dimensional $k$-algebra, $R$ a $k$-ring, and $B$ a
finitely generated $R$-algebra.  By a \textit{ring} we will always mean a commutative associative unital ring. In
contrast, an \textit{algebra} will not be required to be commutative, associative, or unital.  We write $\gen_R(B)$ for
the smallest cardinality of a 
set of generators of $B$ as an $R$-algebra.

Recall that an \textit{$R$-form} of $A$ is an $R$-algebra $B$ such that there exists 
a faithfully flat $R$-ring $S$ for which $A \otimes_k S \cong B \otimes_R S$ as $S$-algebras.
Many important classes of algebras arise as forms of particular $k$-algebras.
For example: 
\begin{enumerate} 
	\item[(i)] $R$-forms of $A = k^n$ with the zero multiplication are projective $R$-modules of rank $n$.
	\item[(ii)] $R$-forms of $A = k^n$ with componentwise multiplication 
	are finite \'etale $R$-algebras of rank $n$. 
	\item[(iii)] $R$-forms of $A = \Mat_s(k)$ are Azumaya algebras of degree $s$ over $R$.
	\item[(iv)] $R$-forms of the split octonion $k$-algebra are octonion algebras over $R$.
\end{enumerate}

\smallskip
O.~Forster~\cite{forster} showed that a projective module of rank $n$ over 
a noetherian ring $R$ can be generated by $\Kdim R + n$ elements, where
$\Kdim R $ denotes the Krull dimension of $R$; see also \cite{swan}.
In~\cite[Theorem~1.2 and Proposition~4.1]{fr}, the first two authors extended this result to algebras as follows.

\begin{thm}[Generalized Forster bound]\label{thm.fr} 
Assume $k$ is infinite. Let $A$ be a finite-dimensional $k$-algebra. If
$R$ is a noetherian $k$-ring and $B$ is an $R$-form of $A$, then
$\gen_R(B) \leqslant \Kdim R + \gen_k(A)$.
\end{thm}

\medskip
In case (i) above, Theorem~\ref{thm.fr} specializes to Forster's bound, since in this case $R$-forms of $A$ are projective $R$-modules of rank $n$.
Forster's bound was shown to 
be optimal by S.~Chase and R.~Swan~\cite[Theorem~4]{swan_counterex}. More precisely, they showed that for every $n \geqslant 1$ and $d\geqslant 0$
there exist a $d$-dimensional ring $R$ of finite type (i.e., generated by finitely many elements) over the field of real numbers 
and a projective $R$-module $B$ of rank $n$ such that $\gen_R(B) = d + n$.

In case (ii), Theorem~\ref{thm.fr} shows that $\gen(B) \leqslant \Kdim R + 1$ for every finite \'etale algebra $B$ over $R$. 
This upper bound was recently shown to be optimal by~A.~Shukla and the third author~\cite{Shukla2020}. 
Specifically, for every pair of integers $n \geqslant 2$ and $d \geqslant 0$, they constructed a finitely generated ring $R$ 
over the field of real numbers, of dimension $d$, and a finite \'etale $R$-algebra $B$ of rank $n$ such 
that $\gen_R(B) = d + 1$.\footnote{A similar example was independently found by M.~Ojanguren (2017, unpublished).} 

In this paper we will focus on the case where $R$ is of finite type  over an infinite field $k$
and will address the following question:

\begin{question} \label{q.main} Is the generalized Forster bound of Theorem~\ref{thm.fr} 
optimal for every finite-dimensional $k$-algebra $A$?
\end{question}

We will show that the answer to this question is ``no", contrary to what one might expect based 
on~\cite{swan_counterex} and~\cite{Shukla2020}. In particular, we will show that the following bounds
given by Theorem~\ref{thm.fr}, 
\begin{enumerate}
\item[(iii)] $\gen_R(B) \leqslant \Kdim R + 2$ for every Azumaya algebra $B$ over  $R$  
and 
\item[(iv)] $\gen_R(B)\leqslant \Kdim R + 3$ for every octonion algebra $B$ over $R$;
\end{enumerate}
(see \cite[Corollary~4.2]{fr}), are far from optimal.

Our first main result is:

\begin{thm} \label{thm.main2} 
Assume that $k$ is an infinite field, $A$ is an $n$-dimensional $k$-algebra, and $R$ is a finite type $k$-ring of Krull dimension $d$. 
Let $\overline{k}$ be an algebraic closure of $k$ and set $\overline{A}=A\otimes_k \overline{k}$. 
If the maximal dimension of a proper $\overline{k}$-subalgebra of $\overline{A}$ is $n_{\rm max}$, then
\[ \gen_R(B) \leqslant \Floor{ \frac{d}{n-n_{\rm max}}} + n_{\rm max} + 1 \]
for any $R$-form $B$ of $A$.
\end{thm}

If $n_{\rm max} \leqslant n - 2$, as in examples (iii) and (iv) above, then Theorem~\ref{thm.main2} yields 
an asymptotically stronger upper bound on $\gen_R(B)$ than Theorem~\ref{thm.fr} (for large $d$).

Our second main result is a lower bound on $\gen_R(B)$ for certain rings $R$ and certain $R$-forms of $A$.
While this lower bound does not quite match the upper bound of Theorem~\ref{thm.main2}, 
it is not far off in the sense that both are linear in $d = \Kdim R$.

\begin{theorem} \label{thm.main3} Let $k$ be a field of characteristic $0$ and let $A$ be an $n$-dimensional $k$-algebra
  such that the group $G=\Aut_{k}(A)$ is not unipotent. Then there exists a positive integer $\rho_G$ with the following property:
  For any positive integer $d$, there exist a finite type $k$-ring $R$ of Krull dimension $d$ and an $R$-form $B$ of $A$ such that 
   \[ \gen_R(B) \geqslant  \Floor{\frac{d + 2 \dim G - \rho_G}{2n} } + 1. \]
Here $\rho_G$ depends only on $G$ and not on 
$A$. Moreover, 
writing $H$ for the quotient of the identity connected component of $G$ by its unipotent radical, we may
take $\rho_G=4$ if $\dim H>0$ and $\rho_G=2$ if $\dim Z(H)>0$.
\end{theorem}

 Theorems~\ref{thm.main2} and~\ref{thm.main3} remain valid in the more general setting where 
we allow $A$ and $B$ to be \emph{multialgebras}, rather than algebras. The same is true for most of the other results of this paper
where the structure of the algebra is not specified. Here by a \emph{multialgebra} over a ring $R$ we mean a module $B$ together 
with a (not necessarily finite) collection of $R$-multilinear maps $\{m_i:B^{n_i}\to B\}_{i\in I}$, where each $n_i$ is a non-negative 
integer.\footnote{This definition
of multialgebra is more restrictive than the one used in \cite{fr}.} For example, a usual binary (algebra) product on $B$ can be specified 
by an $R$-bilinear map $B^2\to B$, an involution by an $R$-linear map $B^1\to B$, and a unit element in $B$ by a map $B^0\to B$ (which is always 
multilinear). A multialgebra structure on $B$ naturally induces one on the tensor product $B \otimes_R S$ for any $R$-ring $S$. Subalgebras,
forms and generating sets of multialgebras are defined in the obvious manner. Viewing $B$ as a multialgebra is useful even in the case 
where $B$ is an $R$-algebra in the usual sense, with a unit element. Here 
we may view $B$ as a unital algebra or a non-unital algebra, and the multialgebra structure of $B$ allows us to distinguish between these two settings.
(See also Remark~\ref{rem.unital} below.)
Aside from usual algebras, unital and non-unital, the multialgebras of greatest interest to us will be algebras with involution.

Our third main result sharpens the upper and lower bounds of Theorems~\ref{thm.main2} and~\ref{thm.main3} in the case of Azumaya algebras.

\begin{theorem} \label{thm.azumaya} Let $k$ be an infinite field.
\begin{enumerate}
\item[(a)] Suppose $R$ is a finite type $k$-ring of Krull dimension $d$, and $B$ is an Azumaya $R$-algebra of degree $s$.
Then $\gen_R(B) \leqslant \Floor{\dfrac{d}{s - 1}} + 2$.
\item[(b)] On the other hand, if $\operatorname{char}(k) = 0$, then
  for any $s \ge 3$ and $d \ge 0$ there exists a finite type $k$-ring $R$ of Krull dimension $d$ 
  and an Azumaya $R$-algebra $B$ of degree $s$ such that $\displaystyle \gen_R(B) \geqslant  \Floor{\dfrac{d}{2(s-1)}} + 2$. 
\end{enumerate}
\end{theorem}

An important role in this paper will be played by the 
automorphism group $G = \Aut_k(A)$ (viewed as a linear algebraic group defined over $k$) 
and by the $k$-schemes $U_r$ and $Z_r$ of $r$-tuples
$(a_1, \ldots, a_r)$ of elements of $A$ that do and do not generate $A$, respectively. It is well known that the category of $R$-forms of $A$ is equivalent to the category
of  $G$-torsors over $\Spec R$, or equivalently, to the category of $R$-points of the classifying stack $BG$. On the other hand, the 
pairs $(B, (b_1, \ldots, b_r))$, consisting of an $R$-form $B$ of $A$ and generators $b_1, \ldots, b_r$ of the $R$-algebra $B$,
are in a natural bijective correspondence with $R$-points of the quotient stack $U_r/G$; see Proposition~\ref{prop.equivariant-map-corres}. 
In fact, $U_r/G$ is an algebraic space; see~Remark~\ref{rem.alg-space}. There is an evident morphism $U_r/G \to BG$ sending 
the pair $(B, (b_1, \ldots, b_r))$ to $B$. An $R$-algebra $B$ can be generated by $r$ elements if and only if the $R$-point of $BG$ corresponding to $B$ can be
lifted to $U_r/G$. Informally speaking, much of this paper revolves around the following question: how well does $U_r/G$ approximate $BG$?
The upper bounds on $\gen_R(B)$ in Theorems~\ref{thm.main2} and~\ref{thm.azumaya}(a) are proved by constructing explicit liftings when
$\Kdim R$ is suitably small. On the other hand, the lower bounds of Theorems~\ref{thm.main3} and~\ref{thm.azumaya}(b) are proved by exhibiting 
topological obstructions to the existence of a such liftings.

The remainder of this paper is structured as follows. In Section~\ref{sect.notations} we collect notational conventions used throughout the paper. 
Sections~\ref{sec:forms-of-algebras-and-gens}-\ref{sect.forster} are devoted to upper bounds on the number of generators.
In Sections~\ref{sec:forms-of-algebras-and-gens} and~\ref{sect.equivariant-maps} we study 
the spaces $U_r$ and $Z_r$ of generators and non-generators of a (multi)algebra $A$.
In Section~\ref{sect.versal} we introduce the notion of a \emph{$d$-versal torsor}. This is a variant of the usual notion of versality 
from~\cite[Section 5]{serre-ci} or~\cite{duncan-reichstein}, where instead of torsors over fields we consider torsors over rings of bounded Krull dimension. In Section~\ref{sect.upper} we prove an upper bound on $\gen_R(B)$ for any $R$-form $B$ of $A$ 
in terms of the codimension of $Z_r$ inside $A^r$ (Theorem~\ref{thm.main}) and use it to deduce Theorem~\ref{thm.main2}. 
In Sections~\ref{sect.azumaya} and~\ref{sect.octonion} we use Theorem~\ref{thm.main} to prove 
new upper bounds on the number of generators of an Azumaya algebra (Theorem~\ref{thm.azumaya}(a)), an Azumaya algebra with involution (Proposition~\ref{prop.involution}) and an octonion algebra (Proposition~\ref{prop.octonion}).
These bounds are of the same form,
\[ \gen_R(B) \leqslant \dfrac{d}{n - n_{\rm max}} + O(1), \]
as the bound given by Theorem~\ref{thm.main2} but they are stronger, because the additive constant is smaller. 
In Section~\ref{sect.forster} we use Theorem~\ref{thm.main} to give a new short proof of the generalized Forster 
bound (Theorem~\ref{thm.fr}) in the case where $R$ is a $k$-ring of finite type.
The second half of the paper, starting from Section~\ref{sect.lefschetz-principle}, is devoted to lower bounds on the number of generators.
In Section~\ref{sect.lefschetz-principle}, we prove a version of the Lefschetz principle that will allow us
to assume that $k$ is a subfield of $\mathbb C$ in these proofs, thus opening the door to the use of non-algebraic techniques  
such as singular cohomology. In Section~\ref{s:equivCohoUr}, we prove Theorem \ref{thm:produceCounterexample}, which is a general device for
constructing examples of $R$-forms $B$ of a $k$-algebra $A$ requiring many generators. We use this device to prove Theorem~\ref{thm.main3} in Section~\ref{sect.main3} and Theorem~\ref{thm.azumaya}(b) in Sections~\ref{sect.matrix-tuples} and~\ref{sect.azumaya(b)}. 
In Appendix~\ref{sect.topological-preliminaries} we give proofs of some results on 
equivariant cohomology of nonsingular varieties that are used in the second half of this paper.
These results are largely known but are not always documented in the literature in the form we need.


\begin{remark} \label{rem.unital} Let $R$ be a $k$-ring and $B$ be an $R$-algebra (in the usual sense) with a unit
  element. Unless otherwise specified, we will view $B$ as a unital algebra, and similarly for algebras with
  involution. Note that if $B_0$ is $B$ viewed as a non-unital algebra and $B_1$ is $B$ viewed as a unital
  algebra, then clearly $\gen_R(B_1) \leqslant \gen_R(B_0) \leqslant \gen_R(B_1) + 1$, and we often have
\begin{equation} \label{e.unital} \gen_R(B_1) = \gen_R(B_0);
\end{equation}
see Lemma~\ref{lem.unital2}.
In particular, equality~\eqref{e.unital} holds for the algebras we consider in Sections~\ref{sect.azumaya} and~\ref{sect.octonion}: Azumaya algebras of degree $\geqslant 2$ with and without involution and octonion algebras. On the other hand, if $B = R$ (viewed as a rank $1$ \'etale algebra over $R$), then~\eqref{e.unital} fails:
$\gen_R(B_0) = 1$, whereas $\gen_R(B_1) = 0$.
\end{remark}

\begin{remark} We prove the lower bounds of Theorems~\ref{thm.main2} and~\ref{thm.azumaya}(b) in characteristic $0$ only.
The main reason is that our arguments rely on the 
Affine Lefschetz Hyperplane Theorem (Theorem~\ref{th:affineLefschetz}) which requires the characteristic-$0$ assumption. 
Since we are forced to work in characteristic $0$, we liberally use non-algebraic techniques 
in our proofs of Theorems~\ref{thm.main3} and~\ref{thm.azumaya}(b). In particular, we work with 
singular cohomology of complex algebraic varieties and treat the complex points of the algebraic space $U_r/G$ 
as a differentiable manifold. 
It is likely that the use of these non-algebraic techniques can be avoided
by replacing singular cohomology with another oriented (as axiomatized in \cite{Panin2003} and \cite{Panin2009}), 
bounded-below, $\Aone$-invariant cohomology theory of smooth varieties, for instance $\ell$-adic \'etale cohomology. 
If one can prove an Affine Lefschetz Hyperplane Theorem for such a cohomology theory, 
then one should be able to establish lower bounds similar to Theorems~\ref{thm.main3} and~\ref{thm.azumaya}(b) in positive characteristic.
\end{remark}

\begin{remark} The results of  Forster \cite{forster} and the first two authors \cite{fr} assume only that $R$ is noetherian, 
whereas Theorem~\ref{thm.main2} and~\ref{thm.azumaya}(a) assume that $R$ is a finite type $k$-ring for some infinite field $k$.
We do not know whether the upper bounds of Theorem~\ref{thm.main2} and Theorem~\ref{thm.azumaya}(a) remain valid for 
an arbitrary noetherian ring $R$.
\end{remark}

\begin{remark}
  The argument we use to establish the lower bounds of Theorem~\ref{thm.azumaya}(b) does not extend to the case of
  quaternion algebras, i.e., when $s=2$. Proposition \ref{prop.XcapY} does not hold when $s=2$ which means
  that our proof of Lemma \ref{pr:MorNotInj} breaks down. In fact, Lemma \ref{pr:MorNotInj} fails for $s = 2$; see
  Remark \ref{rem:prMorNotInjNeeds3}. Nonetheless, the thesis of W.~S.~Gant \cite[Theorem 5.2]{Gant2020} 
  shows that when $s=2$ and $r$ is odd, Lemma~\ref{pr:MorNotInj} continues to hold. Arguing as 
  in the proof of Theorem~\ref{thm.azumaya}(b), we obtain the following:
  \begin{quote}
  If $\operatorname{char}(k) = 0$ and $d \ge 0$, then there exists a finitely generated $k$-ring $R$ of Krull dimension $d$ 
    and an Azumaya $R$-algebra $B$ of degree $2$ such that \[ \displaystyle \gen_R(B) \geqslant 2 \Floor{\frac{d+2}{4}} + 2. \] 
  \end{quote}
  The lower bound on $\gen_R(B)$ here is the same as that of Theorem~\ref{thm.azumaya}(b) when $d$ is congruent to $2$ or $3$
  modulo $4$, and is lower by $1$ when $d$ is congruent to $0$ or $1$.
\end{remark}

\section{Notational conventions}
\label{sect.notations}

Throughout this paper:

\begin{itemize}
\smallskip
\item $k$ will denote a field.
All schemes, varieties, algebraic groups, groups actions and morphisms are assumed to be defined over $k$.
\end{itemize}

By a variety (or equivalently, a $k$-variety) we shall mean a finite-type separated and reduced scheme over $k$. An algebraic group is a group scheme of finite type over $k$.

\begin{itemize}
\smallskip
\item
$A$ will denote a finite-dimensional $k$-algebra, or more generally, a finite-dimensional $k$-multialgebra.

\smallskip
\item
$n$ will denote the dimension of $A$ as a $k$-vector space, unless otherwise stated.

\smallskip
\item
$G$ will denote $\Aut(A)$, the automorphism group  scheme  of $A$ defined over $k$. 
\end{itemize}

By definition, $G$ is a closed subgroup of $\GL(A) \cong \GL_n$. In positive characteristic $G$ need not be smooth. 
If $k$ contains
at least $8$ elements, then
every affine group scheme of finite type arises as the automorphism group scheme of some $k$-algebra;
see \cite{Gordeev2003} in the case $k$ is of characteristic $0$, and \cite{milne-aut} for the general case.
The algebras in question may not be associative, but they are algebras in the
  usual sense (and not just multialgebras). 


\begin{itemize}
\smallskip
\item
$R$ will denote a $k$-ring, i.e., a commutative associative unital $k$-algebra. 

\smallskip
\item
$B$ will denote an $R$-form of $A$. This means that there exists 
a faithfully flat $R$-ring $S$ for which $A \otimes_k S \iso B \otimes_R S$ as $S$-algebras.

\smallskip
\item
$V_1$ will denote the affine space $\mathbb A(A)$ underlying our finite-dimensional $k$-algebra $A$. 
More generally, $V_r$ will denote the $r$-fold  product $V_1 \times_k \ldots \times_k V_1$ for any $r \geqslant 1$. 
If $R$ is a $k$-ring, then $R$-points of $V_r$ are in a natural bijection with $r$-tuples of elements in $A_R = A \otimes_k R$.

\smallskip
\item
$Z_r$ will denote the closed subscheme of $V_r$ consisting of $r$-tuples $a_1, \ldots, a_r \in A$ that do not generate $A$. 
\end{itemize}

More precisely, $Z_r$ is defined as follows.
Denote the multialgebra operations on $A$ by $\{m_i:A^{r_i}\to A\}_{i\in I}$ and recall that $n=\dim_k A$.
Let $W_r$ denote the collection of $\{m_i\}_{i\in I}$-monomials 
in the letters $x_1,\dots,x_r$. 
That is, $W_r= \bigcup_{t=1}^\infty W_{r,t}$, where $W_{r,1}=\{x_1,\dots,x_r\}$ and for any $t > 1$, $W_{r,t}$   
is the set of formal expressions obtained by substituting elements from $W_{r,t-1}$ in the $\{m_i\}_{i\in I}$ in every possible way.
Each $w\in W_r$ determines a morphism $w:V_r\to V_1$, given by substitution.
By definition, $\overline{a}=(a_1,\dots,a_r) \in V_r(k)=A^r$ 
generates $A$ if and only if there exists $\overline{w}  =(w_1,\dots,w_n)\in (W_r)^n$
such that $w_1(\overline{a}),\dots,w_n(\overline{a})$ form a $k$-basis of $A$.
For every $n$-tuple of monomials $\overline{w}=(w_1,\dots,w_n)\in (W_r)^n$, denote by
$f_{\overline{w}}$ the composition of $(w_1,\dots,w_n):V_r\to V_n\cong \Mat_{n\times n}$
with the determinant morphism $\Mat_{n\times n}\to \mathbb{A}^1_k$.
The scheme of non-generators $Z_r$ is defined as the closed subscheme of $V_r$ cut by   the equations $f_{\overline{w}}=0$ as
$\overline{w}$ ranges over $(W_r)^n$.

\begin{remark}
Note that $Z_r$ may not be reduced, and the natural $G$-action on $Z_r$ may not restrict to the associated reduced scheme $Z_r^{\mathrm{red}}$.
For an example, take $r=1$ and let $A$ be the non-unital algebra $k[\varepsilon]/(\varepsilon^2)$ when $k$ is a field of characteristic $2$.
\end{remark}

\begin{itemize}
\smallskip
\item
$U_r=V_r- Z_r$ will denote the open subscheme of $V_r$ consisting of $r$-tuples $a_1, \ldots, a_r \in A$ that generate $A$.
\end{itemize}
The group scheme $G$ acts on $V_r$ from the left by $g\cdot (a_1,\dots,a_r)=(g(a_1),\dots,g(a_r))$, and the action restricts to $Z_r$ and $U_r$. 

\begin{itemize}
\smallskip
\item
$\gen_R(B)$ will denote the minimal cardinality of all sets of elements of $B$ that generate $B$ as an $R$-algebra (or multialgebra).
\end{itemize}

\begin{lemma} \label{lem.unital2}
Let $A$ be a finite-dimensional unital $k$-algebra (in the usual sense, here we do not allow $A$ to be a general multialgebra), 
$R$ be a $k$-ring, and $B$ be an $R$-form of $A$. 
Let $\overline{k}$ be an algebraic closure of $k$ and 
suppose $A\otimes_k\overline{k}$ does not have an augmentation, i.e.,  
there does not exist a morphism  $A \otimes_k \overline{k} \to \overline{k}$ of unital $\overline{k}$-algebras. 
Write $B_1$, resp.\ $B_0$, to denote $B$ considered as a unital, resp.\ non-unital, algebra. Then
$\gen_R(B_1) = \gen_R(B_0)$.
\end{lemma}

\begin{proof}
  The inequality $\gen_R(B_1) \leqslant \gen_R(B_0)$ is obvious, since any generating set for $B_0$ also generates
  $B_1$.  To prove the opposite inequality, we argue by contradiction. Assume the contrary: there exist elements
  $b_1,\dots,b_r\in B$ that generate $B_1$ but not $B_0$. Let $B'$ be the non-unital subalgebra generated by
  $b_1, \ldots, b_r$. Then $B'$ is a proper two-sided ideal of $B_1$ and the $R$-algebra $ B/B'$ is generated as an $R$-module
  by the identity element $1_B$ of $B$. In other words, there a is unital $R$-algebra homomorphism $\epsilon: B\to B/B'\cong R/I$ for some proper ideal $I$ of $R$.
  
  Since we are assuming that $B$ is an $R$-form of $A$, there exists 
  a faithfully flat $R$-ring $S$ such that $B \otimes_R S$ is isomorphic to $A_S = A \otimes_k S$ as an $S$-algebra.
  Tensoring the map $\epsilon: B \to R/I$ with $S$, we obtain a map $\epsilon_S \colon A_S \to S/IS$. Tensoring further with a field $K = S/M$,
  where $M$ is a maximal ideal of $S$ containing $IS$, we obtain an augmentation $\epsilon_K \colon A_K \to K$, where $K$ is a field containing $k$. 
  
  Now let $X$ be the $k$-scheme of 
  unital $k$-algebra homomorphisms from
  $A$ to $k$; it is a closed subscheme of the affine space associated to the dual space of $A$.
  By what we have shown,  $X$ has a $K$-point, namely, the augmentation $\epsilon_K:A_K\to K$. We conclude that $X \neq \emptyset$ and hence, $X$ has a $\overline k$-point. 
  Equivalently $A_{\overline{k}}$, has an augmentation $A_{\overline{k}} \to  \overline{k}$, contradicting our assumption. This contradiction shows that  $\gen_R(B_0) = \gen_R(B_1)$.
%
\end{proof}

All $G$-torsors in this paper
are \emph{left} $G$-torsors $\pi:T\to X$ such that $\pi$ admits a  section after base-changing along some faithfully flat quasi-compact morphism $X'\to X$ (that is, we consider left $G$-torsors for the fpqc topology).

\begin{itemize}
\smallskip
\item
For a $G$-torsor
$T\to \Spec R$ 
(with $G=\Aut_k(A)$ and $R$ a $k$-ring), we let
${}^T\! A$ denote the twist of $A$ by $T$.
\end{itemize}

Recall that ${\,}^T\!A$ is an $R$-form of $A$ defined as follows.
Elements of ${\,}^T \! A$ are $R$-points of the twisted variety ${\, }^T \! V_1$ over $\Spec R$,
which is the quotient of $T\times V_1$ by the  diagonal action of $G$. When $A$ is an algebra (in the usual sense), the $R$-algebra operations on ${\, }^T\! A$ 
are given by twisting each of the following 
operations on $A$ by $T$:
\begin{itemize}
\item[{$\circ$}]
addition, $+ \colon  V_1 \times V _1 \to V_1$, given by $(a_1, a_2) \mapsto a_1 + a_2$;
\item[{$\circ$}]
scalar multiplication $\mathbb A^1 \times V_1 \to V_1$ given by $(c, a) \mapsto c a$; and
\item[{$\circ$}]
algebra multiplication $\cdot \,\colon V_1 \times V_1 \to V_1$ given by $(a_1, a_2) \mapsto a_1 \cdot a_2$.
\end{itemize}
\smallskip
\noindent
Here $G$ acts trivially on $\mathbb A^1$. For details, see \cite[Section III.2.3]{giraud}.
The same construction works for multialgebras.

\smallskip
\begin{itemize}
\item
$\Hoh^i(X)$ will denote the singular cohomology of a topological space $X$ with coefficients in the ring $\ZZ$.

\item
$\Hoh^*(X)$ will denote the graded ring $\bigoplus_{i=0}^\infty \Hoh^i(X)$.
\end{itemize}
If different coefficients are required, they will be specified, e.g., $\Hoh^i(X; \QQ)$. 

\begin{itemize}
\item
If $Y$ is a variety over a field $k$ and $k$ is equipped with an embedding $k \hookrightarrow \C$, then $Y(\C)$ will denote
the topological space underlying the associated analytic space of $Y_\C$,
\cite[Sec.~2]{SerreGeometriealgebriquegeometrie1955}. If $Y$ is such a variety, then notation $\Hoh^i(Y)$ will be taken to
mean $\Hoh^i(Y(\C))$.

\item
If $\Gamma$ is a topological group, then $E \Gamma \to B \Gamma$ will denote a universal principal $\Gamma$-bundle, and
if $\Gamma$ acts on a space $X$, then $\Hoh^i_\Gamma(X)$ will denote the Borel equivariant cohomology group $\Hoh^i((X
\times E\Gamma)/\Gamma)$, \cite[Chapter III]{Hsiang1975}. 
\end{itemize}
The notation $\Hoh^i_G(Y)$ will be used in place of
$\Hoh^i_{G(\C)}(Y(\C))$ when $G$ is a linear algebraic group acting on a variety $Y$ over a field $k \hookrightarrow \C$.

\begin{itemize}
\item
The notation $\Gm$ is reserved for the multiplicative group-scheme $\GL_1$.
\end{itemize}

\section{Preliminaries on schemes of generators and non-generators}
\label{sec:forms-of-algebras-and-gens}

Let $A$ be a finite-dimensional (multi)algebra over a field $k$. Recall from Section~\ref{sect.notations} that we denote the affine space of ordered $r$-tuples of elements in $A$ by $V_r$, the closed subscheme of $r$-tuples not generating $A$  by $Z_r \subset V_r$ and the open subscheme of $r$-tuples of generators by $U_r = V_r - Z_r$. In this section, we will discuss some basic properties of these schemes. We write $A_R$ for $A\otimes_k R$.

\begin{lemma}\label{lem:Ur-sections}
 Let $R$ be a $k$-ring (not necessarily of finite type). Then
 $U_r(R)$ is the set of $r$-tuples $(b_1,\dots,b_r)\in V_r(R)=A_R^r$ that generate $A_R$ as an $R$-algebra.
\end{lemma}

\begin{proof}
  Let $\overline{b}=(b_1,\dots,b_r)\in V_r(R)$ and let $t:\Spec R\to V_r$ be the corresponding morphism.

  Suppose that $\overline{b}$ generates $A_R$. Let $P\in \Spec R$, let $A(P)=A\otimes_k \mathrm{Frac}(R/P)$
  and let $b_i(P)$ denote the image of $b_i$ in $A(P)$. Then $\overline{b}(P):=(b_1(P),\dots,b_r(P))$ generates $A(P)$. By 
  the construction of $Z_r$ in Section~\ref{sect.notations}, there is $\overline{w}\in (W_r)^n$
  such that $f_{\overline{w}}(\overline{b}(P))\neq 0$ in $\mathrm{Frac}(R/P)$. This means that the image 
  of $t$ does not meet the $P$-fibre of $Z_r$. As this holds for all $P\in\Spec R$, the image of $t:\Spec R\to V_r$ 
  is contained $V_r-Z_r=U_r$ and $\overline{b}\in U_r(R)$.
  
  Conversely, assume $\overline{b}\in U_r(R)$. Then for every $P\in \Spec R$, we have
  $t(P)\notin Z_r$. This means that there is  $\overline{w}\in (W_r)^n$
  such that $f_{\overline{w}}(\overline{b}(P))\neq 0$, hence $\overline{b}(P)$ generates $A(P)$. This holds for all $P\in \Spec R$, so by \cite[Lemma~2.1]{fr}, $\overline{b}$ generates $A_R$.
\end{proof}

Recall from Section~\ref{sect.notations} that the group scheme $G=\Aut_k(A)$ acts on $V_r$ from the left, and the action restricts to $Z_r$ and $U_r$.

\begin{lemma} \label{pr:schemeTheoreticallyFree}
The action of $G$ on $U_r$ is scheme-theoretically free, i.e., 
the morphism
\begin{equation}
   \label{eq:5}  \Phi: G \times U_r \longrightarrow U_r \times U_r,
\end{equation}
given by $\Phi(g,(a_1, \ldots, a_r)) = ((g(a_1), \ldots, g(a_r)), (a_1, \ldots, a_r))$ is a closed embedding.
\end{lemma}


\begin{proof} By~\cite[Cor.~18.12.6]{Grothendieck1964},
  showing that $\Phi$ is a closed embedding is equivalent to showing that $\Phi$ is a proper monomorphism.
  To see that $\Phi$ is a monomorphism, it suffices to test that
  \begin{equation}
    \label{eq:8}
    \Phi(R) : G(R) \times U_r(R) \to U_r(R) \times U_r(R)
  \end{equation}
  is injective for every $k$-ring
  $R$. Here, thanks to Lemma~\ref{lem:Ur-sections}, $U_r(R)$ is the set of all $r$-tuples $(b_1, \dots, b_r)$ of generators for $A_R$. 
  The action of $G(R)=\Aut_R(A_R)$ on $U_r(R)$ 
  is free because $g \in G(R)$ fixes $(b_1, \dots, b_r) \in U_r(R)$ if
  and only if $g$ fixes the $R$-algebra generated by the $b_i$, i.e., the entire algebra $A_R$. 
  The element $g$ fixes the entire algebra
  $A_R$ if and only if $g$ is the identity automorphism. 
  It follows that for any $R$, the map $\Phi(R)$ in \eqref{eq:8}
  is an injection, as claimed. 

  It remains to show that $\Phi$ is a proper map. We use the valuative criterion for properness. Suppose $R$ is a
  valuation ring over $k$ with fraction field $K$, and let $\varphi$, $\psi$ be morphisms as in the following
  commutative diagram. We need to show the existence of the dotted arrow.
  \begin{equation*}
    \begin{tikzcd}
      \Spec K \rar{\phi} \dar & G \times U_r \dar{\Phi} \\ \Spec R \rar{\psi} \ar[ur, dashed] & U_r \times  U_r.
    \end{tikzcd}
  \end{equation*}
  We remark that $A_R \subset A_K$, and $U_r(R) \subset U_r(K)$.  The arrow $\phi$ represents an automorphism
  $g \in G(K)$ and an $r$-tuple $(a_1, \dots, a_r) \in U_r(K)$; the arrow $\psi$ represents a pair of $r$-tuples in
  $U_r(R)$, and commutativity of the diagram implies that these $r$-tuples are $(ga_1, \dots, ga_r)$ and
  $(a_1, \dots, a_r)$. In particular, both $(ga_1, \dots, ga_r)$ and $(a_1, \dots, a_r)$ lie in $U_r(R)$. We must show
  that $g$ is defined over $R$. If $e_1, \ldots, e_n$ is a $k$-basis of $A$, we claim that $g(e_i) \in A_R$ for every
  $i = 1, \ldots, n$. This will prove that $g \in \GL_n(R) \cap G(K) = G(R)$.
  
  To prove the claim, note that since $(a_1, \ldots, a_r) \in U_r(R)$,
  it is possible to find polynomials $f_1, \dots, f_n$ in $r$ non-commuting variables with coefficients in $R$
 such that
  \[ f_i (a_1, \dots, a_r) = e_i \quad \forall i \in \{1, \dots, n\}. \]
(When $A$ is a multialgebra, the $f_i$
  should be taken to be formal linear combinations
  of members of $W_r$
  from Section~\ref{sect.notations}.)
  Now consider the result of applying $g$ to $e_i$. Since $g$ is known to lie
  in $G(K)$, $g(e_i)$ is a priori an element of $A_K$, but using the polynomials $f_i$, we see that
  \[ g ( e_i) = g (f_i (a_1, \dots, a_r)) = f_i(g a_1, \dots, g a_r),\]
so $g(e_i)$ actually lies in $A_R$, as claimed. This allows us to produce a lift, so the valuative criterion for properness is satisfied.
\end{proof}

\begin{rmk} \label{rem.alg-space}
  Regard $U_r/G$ as a sheaf on the large fppf site of $\Spec k$. Since $G$ is a linear algebraic group, being by definition a closed subgroup of $\GL(A)$, we may regard $U_r/G$ as an Artin stack.
  Moreover, since the $G$-action on $U_r$ is free, $U_r/G$ is an algebraic space over $k$;
  see~\cite[Tag 04SZ]{stacks_project}.
\end{rmk}

\section{Generators and equivariant maps}
\label{sect.equivariant-maps}

Recall from the introduction that an $R$-form of $A$ is an $R$-algebra $B$ for which 
there exists a faithfully flat $R$-ring $S$ such that $A \otimes_k S \iso B \otimes_R S$ as $S$-algebras. 
It is well known that the category of $G$-torsors over $\Spec R $ is equivalent to the category of $R$-forms of $A$,
functorially in $R$. The equivalence is given by sending an $R$-algebra $B$ to 
$\operatorname{Iso}_{R\text{-}\mathrm{alg}}(A_R, B)$, regarded as an $R$-scheme and 
endowed with the \emph{left} $G_R$-action given by $g\cdot \psi=\psi \circ g^{-1}$ on sections.
Conversely, a $G$-torsor $T \to \Spec R $ corresponds to the twisted algebra ${\,}^T \! A$ defined in Section~\ref{sect.notations}.

\begin{prop} 
\label{prop.equivariant-map-corres}
	Let $R$ be a $k$-ring (not necessarily of finite type), let $B$ be an $R$-form of $A$, and let $T\to \Spec R$ be its associated $G$-torsor.
	Then there is a natural (in $R$) bijective correspondence between the following sets:
	\begin{enumerate}
	\item[(a)] $G$-equivariant $k$-morphisms from $T$ to $U_r$, and
	\item[(b)] $r$-tuples $(b_1,\dots,b_r)\in B^r$ generating $B$ as an $R$-algebra.
	\end{enumerate}	 
\end{prop}

Our proof of Proposition~\ref{prop.equivariant-map-corres} will rely on the following lemmas. 

\begin{lem} \label{lem.torsors} 
Let $X$ and $Y$ be $k$-schemes, $G$ a group scheme over $k$,
and
\[  \xymatrix{ V \ar@{->}[d]_{\alpha}  
&     & W  \ar@{->}[d]^{\beta} \ar@{->}[ll]_{\pi}  \\
 X   
 &      &    Y \ar@{->}[ll]^{\overline{\pi}}}  \]
be a cartesian diagram over $k$, where $\alpha$ and $\beta$ are $G$-torsors and $\pi$ is $G$-equivariant. 
Then there is a natural (in the morphism $X \to Y$) bijection
between sections $\overline{f} \colon X \to Y$ of $\overline{\pi}$ and $G$-equivariant sections $f \colon V \to W$ of $\pi$
such that $\overline{f} \circ \alpha=\beta\circ f$ whenever $f$ corresponds to $\overline{f}$.
\end{lem}

\begin{proof} 
It is enough to prove the corresponding statement for sheaves over the large fpqc site of $\Spec k$.
To that end, it is enough to show that for every $k$-scheme $S$,
there is a unique (and hence natural in $S$, $X$, $Y$) bijection between
$G(S)$-equivariant $\pi(S)$-sections   $f:V(S)\to W(S)$
and $\overline{\pi}(S)$-sections $\overline{f}:X(S)\to Y(S)$ such that 
$\overline{f}\circ \alpha(S)=\beta(S)\circ f$ whenever $f$ corresponds to $\overline{f}$.
That is, it is enough to prove the set-theoretic analogue of the lemma, which is straightforward.
\end{proof} 

\begin{lemma} \label{lem.mor} Let $R$ be a $k$-ring and $T \to \Spec R$ be a $G$-torsor. Then there is a natural isomorphism
between the twisted $R$-(multi)algebra $B = {\, }^T A$ and the $R$-(multi)algebra 
$\Mor_G(T, V_1)$ of $G$-equivariant morphisms $T \to V_1$. Here ``natural" means 
that if $S$ is an
$R$-ring, then this isomorphism is compatible with the base change from $R$ to $S$.
\end{lemma}

The (multi)algebra structure on $\Mor_G(T, V_1)$ is induced from the (multi)algebra structure on $A$. For example, to add $G$-equivariant morphisms
$\phi_1 \colon T \to V_1$ and $\phi_2 \colon T \to V_1$, we compose $(\phi_1, \phi_2) \colon T \to V_1 \times V_1$ with the addition map $+ \colon V_1 \times V_1 \to V_1$.
 
\begin{proof}
Consider the cartesian diagram
\[  \xymatrix{ T \ar@{->}[d]_{\alpha}  &     & T \times V_r  \ar@{->}[d]^{\beta_r} \ar@{->}[ll]_{\pi_r}  \\
  \Spec R    &      &    {\, }^T V_r  \ar@{->}[ll]^{\overline{\pi_r}}  } \]
where $\pi_r \colon T \times V_r \to T$ is projection to the first component. Now set $r = 1$.
A $G$-equivariant morphism $\phi \colon T \to V_1$ gives rise to a $G$-equivariant section $(\id, \phi) \colon T \to T \times V_1$ of $\pi_1$.
The latter descends to a section $b \colon \Spec R  \to {\, }^T V_1$, which we view as an element of $B$. 
This gives us a map $F \colon \Mor_G(T, V_1) \to B$ taking $\phi$ to $b$. By Lemma~\ref{lem.torsors}, $F$ is an isomorphism of sets.

Using the definition of the (multi)algebra structure on $B = {\, }^T A$, one verifies that $F$ is in fact a
homomorphism, and thus an isomorphism, of (multi)algebras.  For example, to show that $F$ is a homomorphism of additive
groups, we consider the diagram:
\[  \xymatrix{ T \ar@{->}[d]_{\alpha}  &     & T \times V_2  \ar@{->}[d]^{\beta_2} \ar@{->}[ll]_{\pi_2} \ar@{->}[rr]^{(\id, +)} &  & T \times V_1 \ar@{->}[d]^{\beta_1}  \\
  \Spec R    &      &    {\, }^T V_2  \ar@{->}[ll]^{\overline{\pi_2}}  \ar@{->}[rr]^{+} & & {\, }^T V_1 } \]
Note that by definition, $V_2 = V_1 \times V_1$, and thus ${\, }^T V_2 = {\, }^T V_1 \times {\, }^T V_1$.  
Similar arguments with $+$ replaced by scalar multiplication and (multi)algebra operations in $A$ show that $F$ is a homomorphism of (multi)algebras. 

Finally, to show that $F$ is natural in $R$, we consider a $k$-ring morphism $R\to S$ and  examine the diagram
\[  \xymatrix{ 
T_S \ar@{->}[d]_{\alpha_S} \ar@{->}[r] & T \ar@{->}[d]_{\alpha}  &     & T \times V_1  \ar@{->}[d]^{\beta_1} \ar@{->}[ll]_{\pi_1}    \\
 \Spec S  \ar@{->}[r] & \Spec R     &      &    {\, }^T V_1  \ar@{->}[ll]^{\overline{\pi_1}, }   } \]
 where the square on the left is Cartesian.
 Since ${\, }^{T_S} V_1 \iso ({\, }^T V_1) \times_{\Spec R } \Spec S $, 
 the $S$-points of the $R$-scheme ${\, }^T V_1$ are in a natural bijection with the $S$-points of the $S$-scheme ${\, }^{T_S} V_1$.  
\end{proof}

\begin{proof}[Proof of Proposition~\ref{prop.equivariant-map-corres}] 
By Lemma~\ref{lem.mor}, an $r$-tuple $b = (b_1, \ldots, b_r)$ of elements of $B = {\, }^T A$ may be viewed as an $r$-tuple of
$G$-equivariant maps $\phi_1, \ldots, \phi_r \colon T \to V_1$ or equivalently, as a $G$-equivariant morphism 
$\phi = (\phi_1, \ldots, \phi_r) \colon T \to V_r = (V_1)^r$. It remains to prove the following.

\smallskip
{\bf Claim.} $b_1, \ldots,b_r$ generate $B$ as an $R$-algebra if and only if $\phi(T)$ is contained in $U_r$.

\smallskip
To prove this claim, let $S/R$ be a faithfully flat ring extension that splits $T$. 
By functoriality ${\, }^{T_S} A = ({\, }^T A) \otimes_R S$. Since $S$ is faithfully flat over $R$, the elements $b_1, \ldots, b_r$
generate ${\, }^T A$ over $R$ if and only if they generate ${\, }^{T_S} A$ over $S$. Without loss of generality, 
we may replace $R$ by $S$ and
thus assume that $T$ is split over $R$, i.e., that $T$ has a section $\psi \colon \Spec R \to T$.
This induces an isomorphism $G\times\Spec R\to T$ given sectionwise by
$(g,x)\mapsto g\cdot \psi(x)$, and an associated algebra isomorphism
$\Mor_G(T,V_1)\iso B= {}^T A\iso {}^{G\times\Spec R}A= A_R$
which, after
unfolding the definitions,
assumes the following simple form: a $G$-equivariant morphism $\phi_i \colon T \to V_1$ corresponds to the element 
$b_i = \phi_i \circ \psi \colon \Spec R  \to V_1$ of $A_R$. By Lemma~\ref{lem:Ur-sections}, $b_1, \ldots, b_r$ generate $B = A_R$ if and only if
the image of 
\[ b = (b_1, \ldots, b_r) = (\phi_1 \circ \psi, \ldots, \phi_r \circ \psi) = \phi \circ \psi \colon \Spec R \to V_r \]
lies in $U_r$. Since $G \times_k \Spec R$ is isomorphic to $T$ via $(g, x) \mapsto g \cdot \psi(x)$, 
$\phi \colon T \to V_r$ is a $G$-equivariant map, and $U_r$ is a $G$-invariant subvariety of $V_r$,
we see that $\phi(T)$ lies in $U_r$ if and only if  $(\phi \circ \psi)(\Spec R) = b(\Spec R)$ lies in $U_r$. This completes the proof of the claim and thus of
Proposition~\ref{prop.equivariant-map-corres}.
\end{proof}

\section{Versal group actions}
\label{sect.versal}

Let $k$ be a field, $G$ be a group scheme over $k$, and $\catC$ be a class of $k$-schemes. 
A $G$-scheme is a $k$-scheme equipped with a $G$-action $G\times_{  k}X\to X$.
We will say that a $G$-scheme $X$ is \emph{versal for $\catC$} if for every
$G$-torsor $\tau \colon T \to Y $ with $Y\in\catC$, there exists a $G$-equivariant $k$-morphism $T \to X$.
We will be particularly interested in the the class $\catC$ 
of affine $k$-schemes  $Y = \Spec R $, where $R$ is a finite type $k$-algebra of Krull dimension $\leqslant d$. 
If $X$ is versal with respect to this class of $k$-schemes, we will say that $X$ is \emph{$d$-versal}.

\begin{remark} The notion of versality has been previously studied in the case where $\catC$ consists of schemes of the form $\Spec K $ for some field $K$
containing $k$; see,~\cite[Section 5]{serre-ci} or~\cite{duncan-reichstein}. Our definition of $d$-versality here is analogous to ``weak versality" 
in~\cite{duncan-reichstein}. For notational simplicity we will use the term ``$d$-versal" instead of ``weakly $d$-versal", even 
though the latter may be more consistent with the literature.
\end{remark}

\begin{prop} \label{prop.uriya's-conjecture}
Let $G$ be an affine algebraic group over an infinite field $k$, $\rho: G \to \GL(V)$ be  a faithful finite-dimensional representation, and
$Z$ be  a closed $G$-subscheme of $V$ of dimension $\leqslant \dim V-(d+1)$ for some  integer $d \geqslant 0$. Then $U=V- Z$
is a $d$-versal $G$-scheme. 
\end{prop}

\begin{remark} \label{rem.approximation}
It is well known that for any finite type affine group scheme $G$ over $k$ and any $d \geqslant 0$, there exists a faithful
representation $G\to \GL(V)$ and a closed $G$-invariant subvariety $Z\subseteq V$ of codimension $> d$ such that $V - Z$ is the
total space of a torsor $(V - Z) \to X$ for some scheme $X$; see~\cite[Lemma~9]{EdidinEquivariantintersectiontheory1998}.
In fact, we may even assume that $X$ is a quasi-projective variety; see~\cite[Remark~1.4]{Totaro}.
In particular, Proposition~\ref{prop.uriya's-conjecture} shows that $d$-versal $G$-schemes with a free action of $G$
exist for every $d \geqslant 0$ and every linear algebraic group $G$.
\end{remark}

\begin{remark}
  Proposition~\ref{prop.uriya's-conjecture} also holds for $Z=\emptyset$ with the convention that $\dim \emptyset = -\infty$. 
  More generally, if $0 \not \in Z$, then $V - Z$ is $d$-versal for every $d \geqslant 0$ simply because every
  $G$-torsor $T\to \Spec(R)$ admits the zero morphism $T\to V$ which is obviously $G$-equivariant. 
\end{remark}

Our proof of Proposition~\ref{prop.uriya's-conjecture} will rely on the following lemma.
Denote the affine space of $a \times b$-matrices over $k$ by $\Mat_{a \times b}$. If $b \geqslant a$,
we will denote the Zariski open subvariety of $a \times b$ matrices of rank $a$ by $\Mat_{a \times b}^{0}$.
We regard $\GL_n$ as an algebraic group over $k$. 

\begin{lem} \label{lem.gln2} Let $G$ be a closed subgroup scheme of $\GL_n$, $R$ be a noetherian $k$-ring, 
and $\tau \colon T \to \Spec R $ be a $G$-torsor. Then there exists an integer 
$m \geqslant n$ and a $G$-equivariant $k$-morphism $F \colon T \to \Mat_{n \times m }^0$.
Here $G$ acts on $\Mat_{n \times m}^0$ via multiplication on the left.
\end{lem}

\begin{proof} We will view the $n$-dimensional $k$-vector space $A = k^n$ as a (non-unital) 
algebra with trivial multiplication, $a_1 \cdot a_2 = 0$ for every $a_1, a_2 \in A$. Then $\Mat_{n \times m}$
is the variety $V_m$ of $m$-tuples of elements of $A$ (each column of $\Mat_{n \times m}$ is identified with $A$),  
$\Mat_{n \times m}^0 $ is precisely the variety $U_m$ of $m$-tuples of generators of $A$, and $\Aut_k(A) = \GL(A) = \GL_n$.
Twisted $R$-forms of $A$ are precisely the projective $R$-modules $B$ of rank $n$. 

It is well known that 
every projective $R$-module of finite rank is generated by finitely many elements. In particular, consider
the $\GL_n$-torsor $T'= T \times^G \GL_n \to \Spec R$ and set $B = {\,}^{T'} \! A$ to be the $R$-module corresponding to $T'$.
Suppose $B$ is generated by $m$ elements. Then by Proposition~\ref{prop.equivariant-map-corres}, there exist
a $\GL_n$-equivariant morphism $F \colon T' \to U_m = \Mat_{n \times m}^0$. Restricting $F$ to $T \subset T'$, we obtain a desired
$G$-equivatiant map $T \to \Mat_{n \times m}^0$.
\end{proof}

\begin{proof}[Proof of Proposition~\ref{prop.uriya's-conjecture}]
Recall that $\rho:G\to \GL(V)$ is a closed embedding~\cite[Theorem~5.34]{milne_alg_groups}.
Let $R$ be a finite type $k$-ring of Krull dimension $\leqslant d$,
and let $\tau \colon T \to \Spec R $ be a $G$-torsor.
Our goal is to establish the existence of
a $G$-equivariant morphism $f:T\to U$.

Let $n=\dim_k V$.
We begin by identifying $V$ with $\Mat_{n\times 1}$, $\GL(V)$ with $\GL_n$ and $G$ with a closed subgroup of $\GL_n$ via $\rho$. 
Observe that $\GL(V) = \GL_n$ acts on $\Mat_{n \times m}$ by multiplication on the left, and so does $G \subset \GL_n$, whereas
$\GL_{m}$ acts on $\Mat_{n \times m}$ via multiplication on the right. These actions commute, and both of them
restrict to $\Mat_{n \times m}^0$. Note also that $\Mat_{n \times m}^0$ is a homogeneous space under the action
of $\GL_{m}$ for all $m \geqslant n$.

Lemma~\ref{lem.gln2} tells us that there exists a $G$-equivariant morphism $F \colon T \to \Mat_{n \times m}^0$ for some integer $m \geqslant n$.
Let $Z'$ be the preimage of $Z \subset V$ under the projection map $p \colon \Mat_{n \times m}^0 \to V$ to the first column.
(Recall that each column of $\Mat_{n \times m}$ is $G$-equivariantly identified with $V$.)
In other words, $Z'$ consists of $n \times m$ matrices of rank $n$ whose first column lies in $Z$.
For $s\in \GL_{m}({k})$, let $r_s:\Mat_{n\times m}^0\to \Mat_{n\times m}^0$ 
denote the $G$-equivariant morphism given by multiplying with $s$ on the right, and let
$sT$ denote $T$, regarded as an $\Mat_{n\times m}^0$-scheme via $r_s\circ F$. 

\smallskip
{\bf Claim.} Let $X$ and $Y$ be $G$-schemes and let $f:X\to\Mat_{n \times m}^0$ and $g:Y\to\Mat_{n \times m}^0 $ be $G$-equivariant morphisms. Suppose that
\[ \dim X + \dim Y < mn + \dim G. \]
Then there exists a non-empty Zariski open subset $W_{X, Y} \subset \GL_{m}$ such that $s X  \times_{\Mat^0_{n\times m}} Y = \emptyset$ for every $s \in W_{X, Y}({k})$.

\smallskip
Assume for a moment that this claim is established. 
Let $f:X\to \Mat_{n\times m}^0$ be the morphism $F:T\to \Mat_{n\times m}^0$ and let $g:Y\to  \Mat_{n\times m}^0$ be the closed immersion $Z'\to \Mat_{n\times m}^0$.
Then 
\[ \dim X = \dim \, T =  \dim  \Spec(R) + \dim G \leqslant d + \dim G, \]
whereas
\[\dim Y = \dim Z + n(m - 1) < (\dim V - d) + n(m -1) = n - d + n(m -1) = mn - d . \]
Thus, $\dim X + \dim Y < mn + \dim G$, and the claim applies to this choice of $X$ and $Y$. 
Since $\GL_{m}$ is rational, and $k$ is an infinite field, $W_{X, Y}$ has a $k$-point $s \in W_{X, Y}(k)$. 
After replacing $F$ by $r_s \circ F$, we may assume that the image of $T$ in $\Mat_{n \times m}^0$ 
is disjoint from $Z'$. Composing $F$ with the projection 
$p \colon \Mat_{n \times m}^0 \to V$ to the first column, 
we obtain a desired $G$-equivariant morphism $f = p \circ F \colon T \to U$. 

\smallskip
It thus remains to prove the claim. Consider the natural projection $\pi \colon I \to \GL_{m}$, where $I \subset \GL_{m} \times X \times Y$ 
is the incidence scheme consisting of triples $(u, x, y)$ such that $f(x)\cdot u = g(y)$. 
We take $W_{X,Y}$ to be the complement of the Zariski closure of $\pi(I)$ in $\GL_{m}$.
By definition, $W_{X, Y}$ is a Zariski open subset of $\GL_{m n}$ defined over $k$. We only need to show that it is non-empty, or equivalently, that $\pi$ is not dominant. Since $\pi$ is quasi-compact, by 
\cite[Proposition~2.3.7]{Grothendieck1965}, 
we may pass to the algebraic closure of $k$ in order to check this, and thus assume that $k$ is algebraically closed.

Since $k$ is algebraically closed,  $G^{\mathrm{red}}$ is a subgroup scheme of $G$ \cite[Corollary~1.39]{milne_alg_groups}.
It is harmless to replace $G$ with the identity connected component of $G^{\mathrm{red}}$, so assume
$G$ is reduced and connected.
This means  that
$X^{\mathrm{red}}$ and $Y^{\mathrm{red}}$ are $G$-varieties \cite[Proposition 2.5.1(1)]{Brion_2017_alg_groups}, so we may
replace $X$ and $Y$ with their reductions. 
We may also assume that $X$ and $Y$ are irreducible. 
Indeed, if $X_1, \ldots, X_{\alpha}$ are the irreducible components of $X$, and $Y_1, \ldots, Y_{\beta}$ are the irreducible components of $Y$,
then 
\[ W_{X, Y} = \bigcap_{i = 1}^{\alpha} \bigcap_{j = 1}^{\beta} \, W_{X_i, Y_{j}}. \]
As $\GL_m$ is irreducible, it is enough to show that each
$W_{X_i,Y_j}$ is nonempty. 
Note that each $X_i$ and each $Y_j$ is a
$G$-variety because $G$ is connected.

Write $M=\Mat^0_{n\times m}$.
At this point,
we may apply the Kleiman Transversality Theorem,
see~\cite[Corollary 4(i)]{kleiman} or~\cite[Theorem 10.8]{HartshorneAlgebraicGeometry1977},
which tells us that there exists an open dense subset $W'$ of $\GL_{m}$ 
such that for every $s\in W'(k)$, 
the fibre product $s X  \times_{M} Y$ is 
either 
\begin{enumerate}
    \item[(i)] empty, or
    \item[(ii)]
    nonempty and equidimensional of dimension $\dim X + \dim Y - \dim M < \dim G$.
\end{enumerate}
Since the $\GL_{m}$-action and the $G$-action on $\Mat_{n \times m}^0$ commute, the scheme $s X  \times_{M} Y$ admits a $G$-action, 
and since $f$ and $g$ are $G$-equivariant, the morphism 
$s X  \times_{M} Y\to \Mat^0_{n\times m}$ is also $G$-equivariant.
The left $G$-action
on $\Mat_{n \times m}^0$ has trivial stabilizers, so the same applies to $s X  \times_{M} Y$.
Thus, if $s X  \times_{M} Y$ is nonempty, then it is of dimension $\geq \dim G$, meaning that  (ii) is impossible.
Consequently, $sX \times_M Y = \emptyset$ for every $s \in W'(k)$, which means that $W'\subseteq W_{X,Y}$. Since $W'$ is dense in $\GL_{m}$, we have $W_{X,Y}\neq\emptyset$, as desired.
\end{proof}

\section{Proof of Theorem~\ref{thm.main2}}
\label{sect.upper} \label{sect.thm.main2}

With Propositions~\ref{prop.equivariant-map-corres} and~\ref{prop.uriya's-conjecture} at hand, we can deduce the following
theorem, from which all of our subsequent upper bounds will follow.
The notation is as in Section~\ref{sect.notations}.

\begin{thm} \label{thm.main} 
Let $A$ be a finite-dimensional algebra over an infinite field $k$,   $r\in \mathbb{N}\cup\{0\}$,
and   $c_A(r)$ be
the codimension of $Z_r$ in $V_r$, i.e.,
$c_A(r)= r \dim_k A - \dim Z_r$. Let $R$ be a $k$-ring of Krull dimension $d$. 
If $c_A(r) > d$, then $\gen_R(B) \leqslant r$ for any $R$-form $B$ of $A$.
\end{thm}

\begin{proof} Assume that $c_A(r)>d$.
Then Proposition~\ref{prop.uriya's-conjecture}, with $V = V_r$ and $Z = Z_r$, tells us that $U_r$ is $d$-versal.
Let $R$ be a $k$-ring of Krull dimension $\leqslant d$, let $B$ be an $R$-form of $A$ and $T \to \Spec R$ be the $G$-torsor 
associated to $B$. Since $U_r$ is $d$-versal, there exists a $G$-equivariant morphism $T \to U_r$, and by Proposition~\ref{prop.equivariant-map-corres}, 
this means that $B$ can be generated by $r$ elements, i.e., $\gen_R(B) \leqslant r$. 
\end{proof}

We now proceed with the proof of Theorem~\ref{thm.main2}.
Let $\overline{k}$ be  an algebraic closure of $k$ and let
$n_{\max}$ be the largest possible dimension of a proper $\overline{k}$-subalgebra of $ A\otimes_k \overline{k}$.
Theorem~\ref{thm.main2} is an immediate consequence of Theorem~\ref{thm.main} in combination with part (b) of the following lemma.

\begin{lemma} \label{lem.main2} (a) $\dim Z_r \geqslant n_{\max} r$. (b) If $r > n_{\max}$, then 
$\dim Z_r \leqslant n_{\max} r + n_{\max}(n-n_{\max})$.
\end{lemma}

\begin{proof} Passing from $k$ to $\overline{k}$ does not change $\dim Z_r$. Hence, we assume that $k$ is algebraically closed.
We will not distinguish between varieties and their sets of $k$-points.

(a) Let $A_{\max}$ be a proper subalgebra of $A$ of maximal dimension $n_{\max}$. If $a_1, \ldots, a_r \in A_{\max}$, then
the subalgebra generated by $a_1, \ldots, a_r$ is contained in $A_{\max}$ and hence the $r$-tuple $(a_1, \ldots, a_r)$ lies in $Z_r$. Thus
$\dim Z_r \geqslant \dim(\underbrace{A_{\max} \times \ldots \times A_{\max}}_{\text{$r$ times}}) = n_{\max} r$.

(b) Any $r$-tuple $a_1, \ldots, a_r \in A$ that spans a linear subspace of dimension $> n_{\max}$ in $A$ 
automatically generates $A$ as a $k$-algebra. Thus,
\begin{equation} \label{e.rank} Z_r^{\mathrm{red}} \subset \operatorname{Rank}_{r, 0}(A) \cup  \Rank_{r, 1}(A) \cup \ldots \cup \Rank_{r, n_{\max}}(A), 
\end{equation}
where here, if $V$ is an $n$-dimensional $k$-vector space, then $\operatorname{Rank}_{r, s}(V)$ stands for 
the subvariety of $V^r$ consisting of $r$-tuples $(v_1, \ldots, v_r)$
such that $\dim \Span_k\{ v_1, \ldots, v_r\} = s$. In our case, $V = A$, but from this point on the algebra structure of $A$ will play no role in the proof.
To compute the dimension of $\Rank_{r, s}(V)$,
let $\Gr(V, s)$ denote  the Grassmannian of
$s$-dimensional subspaces of $V$  and
consider the natural morphism
\[ \Rank_{r, s}(V) \to \Gr(V, s)  \]
sending $(v_1, \ldots, v_r)$ to $\Span_k\{v_1, \ldots, v_r\}$. 
The fibres of this map are readily seen to be irreducible of dimension $r s$;
hence, $\Rank_{r, s}(V)$ is irreducible of dimension
\[ \dim \Rank_{r, s}(V) = \dim \Gr(V, s) + rs = sr+ (n-s)s  
. \] 
In view of~\eqref{e.rank}, it remains to show that
the maximal value of the function \[ f(x) = xr + x(n - x) \] on the interval $[0, n_{\rm max}]$
is attained for $x = n_{\rm max}$. Since $n > n_{\rm max}$ and $r > n_{\rm max}$, we have
\[ f'(x) = r + n - 2x = (n - x) + (r - x) > 0 \]
for any $x \in [0, n_{\rm max}]$, and the desired conclusion follows. This completes the proof of Lemma~\ref{lem.main2} and
thus of Theorem~\ref{thm.main2}.
\end{proof}

\begin{remark} 
\label{rm.upper-bound}
If $nr - \dim Z_r > d$, then by Lemma~\ref{lem.main2}(a), $(n - n_{\max})r > d$ or equivalently, $r > \dfrac{d}{n - n_{\max}}$. 
This shows that the best upper bound we can hope to deduce from Theorem~\ref{thm.main} is of the form
\[ \gen(B) < \frac{d}{n- n_{\max}} + c, \]
where $B$ is an arbitrary $R$-form of $A$, $R$ is a $k$-ring of Krull dimension $d$, and $c$ is a constant 
which depends only on $A$ and not on $d$. In Theorem~\ref{thm.main2}, $c = n_{\max}$. In Sections~\ref{sect.azumaya} and~\ref{sect.octonion} we will find a smaller  
constant term $c$ for specific types of algebras. The linear term $\dfrac{d}{n - n_{\max}}$ in the upper bound of Theorem~\ref{thm.main2}
cannot be sharpened by this method.
\end{remark}

\section{Upper bounds on the number of generators for Azumaya algebras}
\label{sect.azumaya}

Throughout this section,  $\Mat_s$ denotes the affine
$k$-space of $s\times s$ matrices.

\subsection{Proof of Theorem~\ref{thm.azumaya}(a)}
\label{subsect.azumaya}
Our proof will rely on Theorem~\ref{thm.main} with $A = \Mat_s(k)$, the algebra of $s \times s$ matrices over $k$. 
Recall that Azumaya algebras of degree $s$ over $R$ are precisely the $R$-forms of $\Mat_s(k)$.
By definition, $Z_r \subset \Mat_s \times \ldots \times \Mat_s$ ($r$ times)
consists of $r$-tuples $(a_1, \ldots, a_r)$ of $s \times s$ matrices which do not generate $\Mat_s(k)$ as a $k$-algebra.
Our goal now is to compute the dimension of $Z_r$. In order to do so, we may pass to the algebraic closure of $k$ and thus assume that
$k$ is algebraically closed. Under this assumption, we may apply
Burnside's theorem: $Z_r^{\mathrm{red}}$ is the union of the $s-1$ subvarieties $X_1, \ldots , X_{s-1}$ of $\Mat_s^r$ determined by
\begin{equation} \label{e.X_i}
X_i(k) = 
\left\{ (a_1, ..., a_r) \, \bigg| \,
\begin{array}{c}
\text{$a_1, \dots, a_r \in \Mat_s(k)$ have a common}
\\
\text{$i$-dimensional invariant subspace}
\end{array}\right\}. 
\end{equation}
For simple proofs of Burnside's theorem, see~\cite{lam-burnside} or~\cite{rosenthal-burnside}.

\begin{prop} \label{prop.Z} 
(a) For every $i = 1, \ldots, s-1$, the variety
$X_i$ is a closed irreducible subvariety of $\Mat_s^r$ of dimension $r s^2 - (r-1) i (s-i)$.

\smallskip
(b) $\dim Z_r = rs^2 - (r-1)(s-1)$. Equivalently, $c_{\Mat_s(k)}(r) = (r-1)(s-1)$.
\end{prop}

Assume for a moment that Proposition~\ref{prop.Z} is established. 
Then setting $r = \Floor{\dfrac{d}{s - 1}} + 2$, we obtain
\[ c_{\Mat_s(k)}(r) = (r-1)(s-1) = \left(\Floor{\frac{d}{s - 1}} + 1\right) (s-1) > \frac{d}{s - 1} \cdot (s-1) = d.   \]
Theorem~\ref{thm.azumaya}(a) now follows from Theorem~\ref{thm.main}.
It remains to prove Proposition~\ref{prop.Z}. The case where $r = 1$ is obvious, so we will assume that $r \geqslant 2$.

\smallskip
(a) Again, we treat varieties as their sets of $k$-points. Consider the incidence variety
 \[ Y_i = \{(a_1, ..., a_r, W) \,  | \, a_1(W), \ldots, a_r(W) \subset W \}  \subset \Mat_s \times \ldots \times \Mat_s \times \Gr(i, s)  , \] 
equipped with the natural projections
\[ \xymatrix{  & Y_i \ar@{->}[dr]^{\pi_2} \ar@{->}[dl]_{\pi_1} & \\
V_r = \Mat_s^r    &  & \Gr(i, s).  } \]
Here, $\Gr(i, s)$ denotes the Grassmannian of $i$-dimensional subspaces in $k^s$. 
Note that $\pi_2$ is surjective. 
Since $Y_i$ is closed in $V_r\times \Gr(i,s)$ and $G(i,s)$ is projective,
the image of $Y_i$ under $\pi_1$ is closed in $\Mat^r_s$. This image is precisely $X_i$.

The fibre of $W\in \Gr(i,s)$ under $\pi_2$ consists of the $r$-tuples of matrices $(a_1, \ldots, a_r)$
such that $a_i(W) \subset W$ for each $a_i$. This means that in a suitable basis of $k^s$, the matrices $a_1, \dots, a_r$ 
are all block upper-triangular of the form
\[ \begin{bmatrix} A_{i \times i} & B_{i \times (s-i)} \\
0_{(s-i) \times i} & C_{(s-i) \times (s-i)} 
\end{bmatrix} \]
where $A_{i \times i}$ is an $i \times i$-matrix, $B_{i \times (s-i)}$ is an $i \times (s-i)$-matrix, $C_{(s - i) \times (s-i)}$ is an
$(s-i) \times (s-i)$-matrix, and $0_{(s-i) \times i}$ is the $(s-i) \times i$ zero matrix.
In particular, every fibre of $\pi_2$ is an affine space of dimension $r (s^2 - i(s-i))$. By the Fibre Dimension Theorem, 
$Y_i$ is irreducible of dimension $r(s^2 - i(s-i)) + \dim \on{Gr}(i, s) = r s^2 - (r-1) i (s-i)$. 
 
Now recall that $X_i \subset V_r$, is the image of $Y_i$ under $\pi_1$.  Since $Y_i$ is irreducible, so is $X_i$. It remains to show that
$\dim X_i = \dim Y_i$, i.e., the fibre of $\pi_2$ over $(a_1, \ldots, a_r) \in X_i$ in general position is finite. 
This fibre consists of $i$-dimensional subspaces, invariant under each of the matrices $a_1, \ldots , a_r$. If $(a_1, \ldots, a_r) \in X_i$
is in general position, we may assume that $a_1$ has distinct eigenvalues and hence, $a_1$ alone has only finitely many invariant 
subspaces in $k^n$. This shows that the fibre $\pi_2^{-1}(a_1, \ldots, a_r)$ is finite, 
and $\dim X_i=\dim Y_i=
rs^2-(r-1)i(s-i)$.

\smallskip
(b) It follows from part (a) that the components of $Z_r$ of the largest dimension are $X_1$ and $X_{s-1}$.
Thus $\dim Z_r = \dim X_1 = \dim X_{s-1} = rs^2 - (r-1)(s-1)$ and $c_{\Mat_s(k)}(r) = r \dim \Mat_s - \dim Z_r = (r-1)(s-1)$, as claimed.
This completes the proof of Proposition~\ref{prop.Z} and thus of Theorem~\ref{thm.azumaya}(a).
\qed

\begin{cor} \label{cor.azumaya}
Let $R$ be a $k$-ring of finite type and $B$ an Azumaya algebra of degree $s$. 

\smallskip
(a) If $\displaystyle s > 1+ \frac{1}{2}d$, where $d$ is the Krull dimension $R$,
then $\gen_R(B) = 2$.

\smallskip
(b) In particular, $\gen_R(\Mat_m(B)) = 2$ for all but finitely many positive integers $m$. 
\end{cor}

\begin{proof} (a) follows from the inequality $\gen(B) \leqslant \dfrac{d}{s - 1} + 1$ of Theorem~\ref{thm.main3}(a), and (b) is an immediate consequence of (a).
\end{proof}


\subsection{Azumaya algebras with involution}
\label{sect.azumaya-involution}

In the following proposition we consider $R$-algebras with involution $B$, and $\gen_R(B)$ denotes the minimal
cardinality of a generating set of $B$ when regarded as an algebra with involution.

\begin{prop} \label{prop.involution} Assume $k$ is infinite. Let $R$ be a finite type $k$-ring of Krull dimension $d$. 
\begin{enumerate}
\item[(a)] If $B$ is an Azumaya algebra of degree $s$ with orthogonal involution over $R$, then 
 \begin{gather*} \text{\phantom{and,\ }$ \gen_R(B) \leqslant \Floor{\dfrac{d+(s-2)}{2s-3}}+1$, if $s \neq 4$, and} \\
\text{$\gen_R(B) \leqslant \Floor{\dfrac{d+1}{4}}+1$, if $s = 4$.} \end{gather*} 

\item[(b)] If $B$ is an Azumaya algebra of even degree $s$ with symplectic involution over $R$,
 then 
\begin{gather*} \text{$\gen_R(B) \leqslant \Floor{\dfrac{d+(s-1)}{2s-3}}+1$, if $s \geqslant 8$,} \\
\text{$\gen_R(B) \leqslant \Floor{\dfrac{d+6}{9}}+1$, if $s = 6$,} \\
\text{\phantom{and,\ }$\gen_R(B) \leqslant \Floor{\dfrac{d+3}{4}}+1$, if $s = 4$, and}\\
\text{$\gen_R(B) \leqslant d + 2$, if $s = 2$.}
\end{gather*}
\end{enumerate}
\end{prop}

\begin{proof} (a) Recall that Azumaya algebras with an orthogonal involution are forms of $A = (\Mat_s(k), t)$, where
  $t$ is the matrix transposition involution. By~\cite[Proposition 4.10]{usra},
\[ \dim Z_r = \begin{cases} r(s^2 - 2s + 3) + (s-2), & \text{if $s \neq 4$, and} \\
 12r + 1, & \text{if $s = 4$.}
\end{cases} \]
Substituting these formulas into the inequality of Theorem~\ref{thm.main}, we deduce part (a).

(b) is proved by the same argument with $t$ replaced by the standard symplectic involution of $\Mat_s$ (where $s$ is even). In this case the formula 
for $\dim Z_r$ is given by~\cite[Proposition 4.13]{usra}.
\end{proof}

\section{Upper bounds on the number of generators for octonion algebras}
\label{sect.octonion}

An octonion algebra $B$ over $R$ is an $R$-form of the split octonion algebra $A$, as defined in e.g., \cite[Section 2]{Blij1959}.

\begin{prop}
\label{prop.octonion}
Let $k$ be a field and $A$ be the split octonion algebra over $k$. Then
\begin{enumerate}
    \item[(a)] for any $r \ge 3$, the scheme $Z_r$ of $r$-tuples not generating $A$
    (see Section~\ref{sect.notations}) is irreducible of dimension $6r + 5$.
    \item[(b)]
    If $k$ is infinite,
    $R$ is a $k$-ring of Krull dimension degree $d$, and $B$ is an octonion $R$-algebra, then  $\gen(B) \leqslant \Floor{\dfrac{d + 1}{2}} + 3$.
\end{enumerate}
\end{prop} 

Part (b) is an immediate consequence of part (a) and Theorem~\ref{thm.main}. To prove part (a), we may assume without loss of generality that our base field $k$ is algebraically closed. 
Again, we will freely identify varieties with their sets of $k$-points.
Recall that the automorphism group of $A$ is the exceptional group $G_2$.
Since $k$ is algebraically closed, $A$ has no quaternion division $k$-subalgebras, and thus~\cite[Theorem 5]{racine} tells us that,
up to translations by elements of $G_2$, there is a unique maximal subalgebra $A_{\max} \subset A$. 
The subalgebra $A_{\max}$ is $6$-dimensional; it is sometimes called the algebra of 
\emph{sextonions}; 
see~\cite{westbury}. We will use the following description of $A_{\max}$ from~\cite{racine}.
Choose a $k$-basis $x_i, y_i$ for the octonion algebra
$A$, where $i = 0, 1, 2, 3$ and $x_i, y_i$ satisfy the relations \cite[(14) on p.~165]{racine}. Then $A_{\max}$ is spanned by
$x_0, x_1, x_2, y_0, y_1, y_3$. 

\begin{lem} \label{lem.octonion1}
(a) $A_{\max}$ is generated by three elements, 

\smallskip
(b) the stabilizer $H$ of $A_{\max}$ in $G_2$ is a $9$-dimensional parabolic subgroup of $G_2$.
\end{lem}
 
\begin{proof}
(a) It suffices to show that $A_{\max}$ is generated by $x_1$, $x_2$ and $y_1$. Indeed, let $A_0$ be the subalgebra of $A_{\max}$ generated by
$x_1$, $x_2$ and $y_1$. Then the relations~\cite[(14) on p.~165]{racine} tell us that
$x_0 = - x_1 y_1$, $y_0 = 1 - x_0$ and $y_3 = x_1 x_2$ also lie in $B_0$. Thus $A_0 = A_{\max}$, as desired.

(b) Recall that $(x_i, y_i)$ are, by construction,
mutually orthogonal hyperbolic pairs relative to the norm form in $A$ for $i = 0, 1, 2, 3$.
This implies that 
\begin{equation} \label{e.complement} 
\text{$A_{\max}^{\perp} = L$, where $L = \Span_k(x_2, y_3)$, and conversely, $L^{\perp} = A_{\max}$.}
\end{equation}
The multiplication table
in \cite[(14), p. 165]{racine} tells us that $x_2^2 = y_3^2 = x_2 y_3 = 0$. This tells us that $L$ is $\gamma$-isotropic in the sense of D.~Anderson; 
see~\cite[Lemma 3.5]{anderson}. In view of \eqref{e.complement}, the stabilizer $H$ of $A_0$ in $G_2$ equals the stabilizer of $L$ in $G_2$.
By~\cite[Proposition A.5]{anderson}, the stabilizer of $L$ is a $9$-dimensional parabolic subgroup of $G_2$. 
\end{proof}

We are now ready to complete the proof of Proposition~\ref{prop.octonion}(a). Recall
our standing assumptions: $k$ is algebraically closed and $r \geqslant 3$. 

Let $Z_r^{\leqslant t}$ be the subvariety of $r$-tuples $(a_1, \ldots, a_r) \in A^r$ such that $a_1, \ldots, a_r$ generate a $k$-subalgebra 
of $A$ of dimension $\leqslant t$. More precisely, we choose a $k$-vector space basis $e_1, \ldots, e_8$ in $A$ and identify an element 
$x = x_1 e_1 + \ldots + x_8 e_8$ of $A$ with the row vector $(x_1, \ldots, x_8) \in k^8$. Then $Z_r^{\leqslant t}$ is the  Zariski closed subvariety of
$A^r$ cut out by the conditions
\begin{equation} \label{e.Zleqt} \operatorname{rank} \begin{bmatrix} p_1(a_1, \ldots, a_r) \\ \vdots \\ p_{t+1}(a_1, \ldots, a_r) \end{bmatrix} \leqslant t \end{equation}
for any $t+1$ monomials $p_1, \ldots, p_{t+1}$ in $r$ variables. Here $p_1, \ldots, p_{t+1}$ are non-commutative non-associative monomials, $p_i(a_1, \ldots, a_r) \in A$ represents a row of the $(t + 1) \times 8$ matrix in~\eqref{e.Zleqt}, and 
each entry of this matrix is a polynomial function in the coordinates 
of $a_1, \ldots, a_r$.

Now set $\displaystyle Z_r^{(t)} = Z_r^{\leqslant t} - Z^{\leqslant t-1}$ to be the variety 
of $r$-tuples $(a_1, \ldots, a_r)$ such that $a_1, \ldots, a_r$ generate a $k$-algebra of dimension exactly $t$. Since every proper subalgebra of $A$
is of dimension $\leqslant 6$, we have $Z_r = Z_r^{\leqslant 6}$. Recall that every proper subalgebra of $A$ is contained in a $6$-dimensional
subalgebra $S$, which is a translate of $A_{\max}$, and that by Lemma~\ref{lem.octonion1}(a), $A_{\max}$ (and thus $S$) is generated by $3$ (and thus $r$)
elements as a $k$-algebra. This tells us that
$Z_r^{(6)}$ is dense in $Z_r$.

It remains to show that $Z_r^{(6)}$ is an irreducible variety of dimension $6r + 5$. Indeed, consider the natural
morphism $f \colon Z_r^{(6)} \to \Gr(A, 6)$ taking an $r$-tuple $(a_1, \dots, a_r)$ 
to the $6$-dimensional subalgebra generated by $a_1, \ldots, a_r$. To see that $f$ is a morphism, note that $Z_r^{(6)}$ is covered by open subsets
\[ U_{p_1, \ldots, p_6} = \{ (a_1, \ldots, a_r) \, | \, p_1(a_1, \ldots, a_r) \wedge \ldots \wedge p_6(a_1, \ldots, a_r) \neq 0 \; \; 
\text{in} \; \; \Lambda^6(A) \} \]
indexed by $6$-tuples of monomials $p_1, \ldots, p_6$ in $r$ variables. On any $U_{p_1, \ldots, p_6} \neq \emptyset$, we define 
\begin{align*} f_{p_1, \ldots, p_6} \colon U_{p_1, \ldots, p_6} &\to \Gr(A, 6) \\
(a_1, \ldots, a_r) &\mapsto p_1(a_1, \ldots, a_r) \wedge \dots \wedge p_6(a_1, \ldots, a_r) \in \bbP(\Lambda^6(A)). 
\end{align*} 
Here we identify $\Gr(A, 6)$ with the variety of decomposable tensors in $\bbP(\Lambda^6(A))$ via the Pl\"ucker embedding. One readily sees from this definition that every $f_{p_1, \ldots, p_6}$ is a morphism. These morphisms agree on the overlaps: the image of $(a_1, \ldots, a_r) \in Z_r^{(6)}$ 
under any of them is the $6$-dimensional subalgebra generated by $a_1, \ldots, a_r$. Hence, they 
patch together to a morphism $f \colon Z_r^{(6)} \to \Gr(A, 6)$. By Lemma~\ref{lem.octonion1}(b), the image of $f$ is the irreducible 
projective homogeneous space $G_2/H$ of dimension $14-9 = 5$. The fibres of $f$ are isomorphic to Zariski open
subvarieties of $V_r$.
By Lemma~\ref{lem.octonion1}(a), these open subvarieties are not empty, and so they are irreducible of dimension $6r=r\dim A_{\max}$. Applying the Fibre Dimension Theorem to
$f$, we conclude that $Z_r^{(6)}$ is irreducible of dimension $6r + 5$, as claimed. This completes the proof of Proposition~\ref{prop.octonion}.
\qed

\section{A new proof of the generalized Forster bound}
\label{sect.forster}

As yet another application of Theorem~\ref{thm.main} we will now
give a short proof of the following variant of Theorem~\ref{thm.fr}.

\begin{thm} \label{thm.fr1} Assume that $k$ is infinite field, $A$ is a finite-dimensional $k$-(multi)algebra, 
$R$ is a finite type $k$-ring of Krull dimension $d$, and $B$ is an $R$-form of $A$. Then
$\gen_R(B) \leqslant d + \gen_k(A)$.
\end{thm}
 
Our proof will rely on the following.

\begin{prop}\label{prop:cAr-increasing}
$c_A(r+1) \geqslant c_A(r) + 1$ whenever $r\geqslant \gen_k(A)-1$.
\end{prop}

As a consequence of Proposition~\ref{prop:cAr-increasing}, we see that $c_A(r) \geqslant r -  \gen_k(A) + 1$ for any $r \geqslant \gen_k(A)$. Setting $r = d + \gen_k(A)$, we can rewrite this as $c_A(r) > d$. 
Theorem~\ref{thm.main} now tells us that $\gen_R(B) \leqslant r = d + \gen_k(A)$, and Theorem~\ref{thm.fr1} follows.

In the course of proving Proposition~\ref{prop:cAr-increasing} we may assume without loss of generality that
$k$ is algebraically closed. It is then harmless to replace $V_r$ and $Z_r$ by their sets of $k$-points endowed with the Zariski topology, 
and this will be our convention in this section.

We say that an $r$-tuple $(a_1,\dots,a_r)\in A^r$ \emph{almost generates} $A$ if there exists
$a_{r+1}\in A$ such that $a_1,\dots,a_r,a_{r+1}$  generates $A$.
Denote the set of almost generating $r$-tuples in $A^r$ by $U'_r$.

\begin{lem}\label{lem:almost-gens-fibre}
	Let $\psi:V_{r+1}\to V_r$ denote projection to the first $r$ coordinates.
	
	\smallskip
	(a) Let $\overline{a}=(a_1,\dots,a_r)\in Z_r$. Then
$\dim \big( \psi^{-1}(\overline{a}) \cap Z_{r + 1} \big) <\dim A$ 
	if and only if $\overline{a}\in U'_r$.
	
	\smallskip
	(b) 
	$U_r'$ is Zariski open in $V_r$.
	
	\smallskip
	(c) If $r\geqslant \gen_k(A) - 1$, then  $Z_r\cap U'_r$ is dense in $Z_r$.
\end{lem}

\begin{proof} (a) is obvious from the definition of $U_r'$. 
(b) holds because $\psi$ is flat and finitely presentated,
hence open, and $U'_r=\psi(U_{r+1})$.

(c)	It is enough to show that if there exists $\overline{a}=(a_1,\dots,a_r)\in Z_r-\overline{ Z_r\cap U'_r }$,
	then $r<\gen_k(A)-1$, or equivalently, that  $U'_r=\emptyset $.	To prove this assertion, we will show by induction on $i=0,1,\dots,r$ that the affine
	space
\begin{equation} \label{e.Li} 
\text{$
	L_i:=\{(x_1,\dots,x_{i },a_{i+1},\dots,a_r)\,|\, x_1,\dots,x_{i+1}\in A\}  
	$	is contained in  $Z_r - U'_r$.}
\end{equation}
For $i = r$,~\eqref{e.Li} tells us that $\emptyset = L_r \cap U_r' = V_r \cap U_r' = U_r'$.
	In other words, $r < \gen_k(A) - 1$, and part (c) follows.
	
	The base case, where $i=0$, is clear. For the induction step, we assume
	assume that $L_i\subseteq Z_r-U'_r$ for some $i\in\{0,\dots,r-1\}$.
	This means that for every $x_1, \ldots, x_r \in A$, $(x_1,\dots,x_i,a_{i+1},\dots,a_r) \notin  U'_r$.
	Consequently, for every $x_{i+1} \in A$, the $r+1$ elements $x_1,\dots,x_{i+1},a_{i+1},\dots,a_r$
	do not generate $A$ as an $k$-algebra, and, in particular, $(x_1,\dots,x_{i+1}, a_{i+2},\dots,a_r)\in Z_r$.
	In other words, $L_{i+1}\subseteq Z_r$. Our goal is to show that $L_{i+1}\cap U'_r = \emptyset$. Assume the contrary.
	Then  by part (b),  $L_{i+1}\cap U'_r$ is Zariski dense	in $L_{i+1}$. In particular,
	$\overline{a}$ lies in the Zariski closure of  $Z_r \cap U_r'$, contradiction 
	our assumption that $\overline{a}\notin Z_r - \overline{Z_r\cap U'_r}$. This contradiction tells us that
	$L_{i+1}\cap U'_r=\emptyset$, completing the proof of~\eqref{e.Li} and thus of part (c).
\end{proof}

\begin{proof}[Proof of Proposition~\ref{prop:cAr-increasing}] 
	By Lemma~\ref{lem:almost-gens-fibre} the fibre of 
	$\psi:Z_{r+1}\to Z_r$ over every point of a dense open subset $Z_r \cap U_r'$ of $Z_r$ is of dimension $< \dim(A)$.
	Hence, $\dim Z_{r+1}<\dim Z_r +\dim A$ by the Fibre Dimension Theorem. Equivalently, $c_A(r+1) > c_A(r)$.
\end{proof}

\begin{remark} The above argument is conceptually simpler than the proof given in~\cite[Section 3]{fr}. It does not entirely supplant this proof,
though, because the main result of~\cite{fr} is more general than Theorem~\ref{thm.fr1}.
\end{remark}

\begin{example} \label{ex.vector-bundle} 
Let $A = k^n$, viewed as a $k$-module (i.e., as a $k$-algebra with zero multiplication). Here $R$-forms of $A$ are projective $R$-modules
of rank $n$ and $G = \Aut_k(A) = \GL_n$. 
Let us write elements of $V_1$ as column vectors of length $n$ and identify $V_r$ with the space $\on{M}_{n \times r}$ of $n \times r$ matrices.
We  assume that $r \geqslant n$; indeed, we need at least $n$ elements
to generate $B$. The group $G = \GL_n$ acts on $V_r = \on{M}_{n \times r}$ via left multiplication. Clearly
$U_r = \on{M}_{n \times r}^0$ is the open subvariety of $n \times r$ matrices of rank $n$, and $Z_r$ is the closed subvariety
of $\on{M}_{n \times r}$ of matrices of rank $< n$, i.e., of matrices with linearly dependent rows.
An easy calculation shows that $\dim Z_r = (n-1) + (n-1)r = (n-1)(r+1)$, and thus $c_A(r) = r \dim A - \dim Z_r = r - n + 1$.
By Theorem~\ref{thm.main}, if $R$ is of finite type over $k$ and $r - n + 1 > \Kdim R$, then $B$ is generated by $r$ elements as an $R$-module.
In other words, $\gen(B) \leqslant \Kdim R + n$, which is Forster's original bound.
\end{example}

 \begin{example} \label{ex.etale}
 Let $A = k \times \ldots \times k$ ($n$ times) with componentwise multiplication. Here $R$-forms of $A$ are \'etale algebras of rank $n$ over $R$ and
 $G = \Aut_k(A)$ is the symmetric group $\on{S}_n$ permuting the $n$ factors of $k$. Once again, we will write elements of $A$ as column vectors of length $n$, 
 and elements of $V_r$ as  $n \times r$-matrices. Then $\on{S}_n$ acts on $V_r = \on{M}_{n \times r}$ by permuting the rows, 
 $U_r$ consists of matrices with distinct rows, and $Z_r$ consists of matrices whose rows are not distinct. 
  From this we readily see that $\dim Z_r = nr - r$ or equivalently, $c_{A}(r) = rn  - \dim Z_r = r$. 
  If $R$ is a $k$-ring of finite type, Theorem~\ref{thm.main} tells us that if $r > \Kdim R$, 
  then every \'etale $R$-algebra $B$ is generated by $r$ elements. 
  Equivalently, $\gen_R(B) \leqslant \Kdim R + 1$, which is the upper bound of Theorem~\ref{thm.fr}.
 \end{example}

\section{The Lefschetz principle}
\label{sect.lefschetz-principle}

Our proofs of Theorems~\ref{thm.main3} and~\ref{thm.azumaya} will rely on topological methods. These methods work best
if the base field $k$ is equipped with an embedding $k \into \mathbb C$. To extend our arguments to an arbitrary base
field $k$ of characteristic $0$, we will repeatedly use the following version of the Lefschetz principle.

\begin{lemma} \label{lem:requireForLefschetz} Let $R$ be a $k$-ring, and $B$ be an $R$-algebra. Suppose $B$ is finitely generated as an $R$-module.
For any field extension $K/k$, set $R_K = R \otimes_k K$ and $B_K = B \otimes_R R_K$. Then
\begin{enumerate}[label=(\alph*)]
\item \label{le1} $\gen_{R}(B) \ge \gen_{R_K}(B_K)$.
\item \label{le2} Moreover, there exists a subextension $k \subset F \subset K$ such that $F/k$ is finitely generated and
  $\gen_{R_F}(B_{F}) = \gen_{R_K}(B_K)$.
\item \label{le3} Suppose $k$ is a finitely generated extension of $\QQ$. If $\gen_{R_K}(B_K) < r$ for some field $K$ containing $k$, then
\[ \gen_{R_{\mathbb C}}(B_{\mathbb C}) < r.\]
\end{enumerate}
\end{lemma} 

\begin{proof} \ref{le1} is obvious: if $a_1, \ldots, a_{r}$ generate $A$ as an $R$-algebra, then $a_1, \ldots, a_{r}$ generate $B_K$ as an $R_K$-algebra.

  \smallskip \ref{le2} Suppose $\gen_{R_K}(B_K) = r$. Choose $r$ elements, $b_1, \ldots, b_{r} \in B_K$ that generate $B_K$ as
  an $R_K$-algebra. We claim that these same elements will generate $B_F$ as an $R_F$-algebra for some extension
  $k \subset F \subset K$ such that $F$ is finitely generated over $k$. In particular, the claim tells us that
  $\gen_{R_F}(B_{F}) \leqslant r$. The opposite inequality is given by part \ref{le1}, so if we can prove this claim, then
  part \ref{le2} will follow.

To prove the claim, choose elements $a_1, \ldots, a_d$ that generate $B$ as an $R$-module and
write each element $b_i$ as 
\[ b_{i} = r_{i1} a_1 + \ldots + r_{id} a_d. \] for some coefficients $r_{ij} \in R_K$. Each $r_{ij}$ lies in
$ R_{F_{ij}}$ for some intermediate subfield $k \subset F_{ij} \subset K$ such that $F_{ij}$ is finitely generated over
$k$.  After replacing $k$ by the compositum of $F_{ij}$ in $K$ (which is still finitely generated over $k$), we may
assume without loss of generality that $b_1, \ldots, b_{r}$ lie in $B$.

Since $b_1, \ldots, b_{r}$ generate $B_K$ as an $R_K$-algebra, we can write each $a_i$ as
\[ a_i = s_{i1} M_{i1} + \ldots + s_{it_i} M_{it_i} \] for some monomials $M_{ij}$ in $b_1, \ldots, b_{r}$ and some
coefficients $s_{ij} \in R_K$. Once again, each $s_{ij}$ lies in some intermediate subfield $E_{ij}$ finitely generated
over $k$. Setting $F$ to be the compositum of $E_{ij}$ in $K$, we see that the $R_F$-subalgebra of $B_F$ generated by
$b_1, \ldots, b_{r}$ contains $a_1, \ldots, a_d$. This shows that $b_1, \ldots, b_{r}$ generate $B_F$ as an
$R_F$-algebra. Hence, $\gen_{R_F}(B_F) \le r $, as desired.

\ref{le3} Choose $F$ as in part \ref{le2}, so that $\gen_{R_F}(B_F) < r$. Since $F$ is a finitely generated extension of
$\mathbb{Q}$, it is isomorphic (over $k$) to a subfield of $\mathbb{C}$. Thus we may assume without loss of generality
that $F \subset \mathbb C$. By part \ref{le1},
 \[ \gen_{R_{\mathbb C}}(B_{\mathbb C}) \leqslant \gen_{R_F}(B_F) = \gen_{R_K}(B_K) < r, \]
as desired.
\end{proof}

\section{Equivariant cohomology of  \texorpdfstring{$U_r$}{Ur}}
\label{s:equivCohoUr}

Suppose the field $k$ is embedded in $\C$, so that a $k$-variety $X$ gives rise to a complex analytic space $X(\C)$. The
purpose of this section is to calculate some of the $\PGL_s(\C)$-equivariant cohomology of the variety $U_r(\C)$. We
begin with two short subsections that cover preliminary material.

We write $\Hoh^*(X)$ for the singular cohomology $\Hoh^*(X(\CC); \Z)$. 

\subsection*{The Affine Lefschetz Theorem}

We use singular cohomology, rather than some other cohomology theory of varieties, because of the following theorem.

\begin{theorem}[Affine Lefschetz Hyperplane Theorem] \label{th:affineLefschetz}
Let $k \subset \CC$ be a field with a chosen complex embedding, let $\mathbb{A}^N$ be an affine space and
$X \subset \mathbb A^N$ be a smooth closed subvariety of dimension $d$. Then there exists a smooth affine hyperplane section $Y
\subseteq X$ defined over $k$ such that $\dim Y=d-1$ and $\Hoh^i(X) \to \Hoh^i(Y)$ is an isomorphism if $i < d-1$ and an injection if $i
= d-1$.
\end{theorem}

Here, a hyperplane section is the scheme-theoretic intersection of $X$ with a hyperplane in the ambient
$\A^N$. 

When $k = \CC$, the theorem is an immediate consequence of Bertini's theorem and \cite[Theorem
1.1.3]{HammLefschetztheoremsquasiprojective1985}. In fact, if $k$ is a field with a chosen complex embedding
$k \subseteq \CC$, then we can still apply these results
. There is an open dense set of complex hyperplanes in some projective space for which the intersection $Y$ is smooth
and such that $Y \to X$ has the required behaviour in cohomology. Any open dense subset of an affine space contains a
point defined over $\QQ$, and therefore over $k$.


\subsection*{Equivariant cohomology of varieties}

If $\Gamma$ is a Lie group, then there exists a universal principal $\Gamma$-bundle $E\Gamma \to B\Gamma$ where $E\Gamma$ is contractible. If $\Gamma$
acts on $X$, then the \textit{Borel equivariant cohomology} (\cite[III.1]{Hsiang1975}) of $X$ is defined by
\[ \Hoh^*_\Gamma( X  ) := \Hoh^*( (X \times E\Gamma)/\Gamma ; \ZZ ). \]
As a special case, if $X$ is a point, then $\Hoh^\ast_\Gamma(X) = \Hoh^\ast(B \Gamma ; \ZZ)$.

If $G$ is a linear algebraic group defined over $k \subseteq \CC$, then we may define a Lie group $G(\CC)$. Suppose $G$
acts on a $k$-variety $X$. We abuse notation and write $\Hoh^*_G(X)$ for $\Hoh^*_{G(\CC)}(X(\CC))$. The ideas of
B.~Totaro~\cite{Totaro} and D.~Edidin and W.~Graham~\cite{EdidinEquivariantintersectiontheory1998} allow 
us to treat $\Hoh^*_G(X)$ as though it were the cohomology of a variety, as we now recall.

For any $n$, we can
find a $G$-representation $V$ and an open $G$-invariant subvariety $U\subset V$ with the properties that $U \to U/G$ is
a $G$-torsor, $U/G$ is a variety, and $(V-U) \hookrightarrow V$ is a closed subvariety of codimension exceeding
$n/2$; see Remark~\ref{rem.approximation}. Then
  \begin{equation}
    \label{eq:11}
    \Hoh^i_{G}(X) \iso \Hoh^i((X \times U)/G) 
  \end{equation} 
  for $i\le n$.

  If $G$ acts on a variety $X$ in such a way that $X \to X/G$ is a $G$-torsor, then the quotient $(X \times U)/G$
  appearing in \eqref{eq:11} above is isomorphic to $X/G \times U$. We deduce that in this case
  \[ \Hoh^n_G(X) \iso \Hoh^n(X/G) \]
  for all nonnegative integers $n$.

    The equivariant Borel cohomology groups for $G$ are contravariantly functorial for
  $G$-equivariant maps. Since there is a unique $G$-equivariant map $X \to \Spec \CC$, there is a natural map
  $\Hoh^*_G(\Spec \C) = \Hoh^*(BG) \to \Hoh^*_G(X)$.

\subsection*{The case of \texorpdfstring{$U_r$}{U(r)}}

The purpose of this section is to prove Theorem \ref{thm:produceCounterexample}, which is a general device for
constructing examples of forms of an algebra $A$ requiring many generators.

We continue to write $V_r$ for the affine variety whose $k$-points are $A^r$, $Z_r$ for the closed subscheme of $V_r$
representing $r$-tuples that fail to generate $A$ as an $k$-algebra, and $c_A(r)$ for the codimension of $Z_r$ in
$V_r$. As before, write $U_r= V_r \sm Z_r$ and
$c_A(r)=\dim V_r-\dim Z_r$. The automorphism group of the algebra $A$ is denoted $G$.

\begin{lemma} \label{lem:isoRange}
  The natural map $\Hoh^n(BG) \to \Hoh^n_G(U_r)$ is an isomorphism for all values of $n < 2c_A(r)-1$.
\end{lemma}
\begin{proof}
  The inclusion $U_r \to V_r$ induces an isomorphism
  \[  \Hoh^n(BG) \iso \Hoh^n_G(V_r) \to \Hoh^n_G(U_r) \]
  for  $n < 2c_A(r)-1$
  by virtue of Lemma \ref{lem:localizationIso} and homotopy invariance for $\Hoh^n_G(\cdot)$.
\end{proof}

  Suppose $R$ is a $k$-ring of finite, $B$ is an $R$-form of the algebra $A$, and  $(b_1, \dots, b_r) \in B^r$ is a generating $r$-tuple of elements. 
  Let $T \to X$ be the $G$-torsor associated to $B$. There is a $G$-equivariant classifying map $\phi_r:
  T \to U_r$ by Proposition~\ref{prop.equivariant-map-corres}, which induces a morphism
  \begin{equation}
    \label{eq:2}
     \phi_r^* : \Hoh^*_G(U_r) \to \Hoh_G^*(T) = \Hoh^*(\Spec R).
  \end{equation}
  This morphism is compatible with the natural maps from $\Hoh^*(B G)$ to $\Hoh^*_G(U_r)$ and $\Hoh^*_G(T)$. We have therefore proved:


\begin{lemma} \label{lem:prohibitGeneration} Suppose $X=\Spec R$ is an affine variety and $B$ is an $R$-form of $A$ with associated
  $G$-torsor $T$. Suppose that $B$ can be generated by $r$ elements. Then there is a commutative diagram of
  cohomology rings
  \[ \xymatrix{ & \Hoh^\ast(BG) \ar[dl] \ar[dr] \\ \Hoh^*_G(U_r) \ar[rr] & & \Hoh^*_G(T) = \Hoh^*(X) } \]
  in which the maps with source $\Hoh^\ast(BG)$ are the natural maps for $G$-equivariant cohomology.
  \qed
\end{lemma}



  If $X$ is a regular $k$-variety, then there exists a vector bundle $E$ on $X$ and an $E$-torsor $p : W \to X$ for which $W$ is an
  affine scheme. This version of the Jouanolou construction is due to Thomason, and is proved in greater generality in \cite[Proposition 4.4]{WeibelHomotopyalgebraictheory1989}.

The vector bundle $E$ is to be viewed as a group-scheme on $X$ where the group structure is given by addition in the
vector bundle. When the construction is applied to a regular variety $X$, the result is an affine $k$-variety
$Y$ such that there is a map $p : Y \to X$ that is Zariski-locally isomorphic to the projection $U \times \A^n_k \to
U$. In particular, the map $p: Y(\C) \to X(\C)$ is a fibre bundle with contractible fibres, and is therefore a homotopy
equivalence, \cite[Theorem 7.57]{James1984}.

\begin{theorem} \label{thm:produceCounterexample}
  Let $k$ be a field with a chosen embedding $k \hookrightarrow \CC$ and $A$ be a finite-dimensional algebra over $k$.
  Suppose the natural map $\Hoh^i(BG) \to \Hoh^i_G(U_r)$ is not injective for some $r \geqslant \gen_k(A)$.
  Then there exists a finite type $k$-ring $R$ and an $R$-form $B$ of $A$ so that $\Kdim R = i$ and $\gen_{R_{\CC}}(B_\CC) > r$.
\end{theorem}

In particular, $\gen_{R}(B) > r$; see Lemma~\ref{lem:requireForLefschetz}(a).

\begin{proof} Recall that by Proposition~\ref{prop:cAr-increasing} 
$c_A(r + t) \to \infty$ as $t \to \infty$. Choose a sufficiently large positive integer $t$ so that 
\begin{equation} \label{e.choice-of-t} 2c_A(r + t)-1 > i \quad \text{and} \quad (r+t) \dim(A) \geqslant i + \dim(G). 
\end{equation}
The algebra $B$ will be defined as ${\, }^T A$, for a suitable $G$-torsor $T \to \Spec R$.
Our goal is to construct this torsor. We will do so by starting with the $G$-torsor $U_{r + t} \to U_{r+t}/G$ 
and modifying it in stages. At each stage, we will produce a $G$-torsor $Y \to Y/G$ so that $Y$ has 
the same $G$-equivariant cohomology as $U_{r+t}$ in degrees less than or equal to $i$.

  The first problem with $U_{r+t}$ is that $U_{r+t}/G$ may not be a scheme, only an algebraic space; see Remark~\ref{rem.alg-space}.
  To fix this, recall that by Remark~\ref{rem.approximation} there exists a linear representation $G \to \GL(\bar W)$
  and a $G$-invariant open subset $W \subset \bar{W}$ such that 
  $\bar W \sm W$ has codimension at least $i/2$ and $W$ is the total space of a $G$-torsor $W \to W/G$ in the
  category of schemes. There is now a torsor $q':U_{r + t} \times W \to (U_{r + t} \times W)/G$ in the category of schemes; see
  \cite[Prop.~23(1)]{EdidinEquivariantintersectiontheory1998} for instance. The quotient scheme $(U_{r + t} \times W)/G$ is
  separated \cite[Lemma 0.6]{Mumford1994} and quasi-compact, being the surjective image of a
  quasi-compact scheme. It is of finite type by \cite[Prop.~1.5.4(v)]{Grothendieck1964},
  since $U_r \times W$ is of finite type. The morphism $U_{r + t} \times W \to (U_{r + t} \times W)/G$ is smooth, since $G$ is
  smooth, and because $U_{r + t} \times W$ is a smooth $k$-variety, it follows from
  \cite[Prop.~17.7.7]{Grothendieck1967} that $(U_{r + t} \times W)/G$ is a smooth $k$-variety.

  The codimension of $\bar W \sm W$ in $W$ is at least $i/2$, so that there is an isomorphism $\Hoh^j_G(U_{r+t}) \to
  \Hoh^j_G(W \times U_{r+t})$ for all $j \le i$  (see Lemma~\ref{lem:localizationIsoEquiv}).

  The next problem is that $(W \times U_{r+t})/G$ is not affine in general. We will approximate it by an affine scheme as follows.
  Write $f: \Spec R \to (W \times U_{r+t})/G$
  for an affine vector-bundle torsor using the Jouanolou construction
  , and consider the pullback
  \begin{equation}
    \label{eq:9}
    \begin{tikzcd}
    T \arrow{r}{F} \arrow{d}{q} \arrow[dr, phantom, "\lrcorner", very near start]  & W \times U_{r+t} \arrow{d} \\
    \Spec R \arrow{r}{f} & (W \times U_{r+t})/G.
  \end{tikzcd}
  \end{equation}
  The composite of $F$ with the projection map $W \times U_{r+t}  \to U_{r + t}$ is
  $G$-equivariant and classifies an $R$-form $B$ of $A$, which is generated by $r+t$ elements.\benw{Revisit this after
    the ``associated forms'' part has been finalized.} The map $f$ is a vector-bundle
  torsor, and therefore induces an isomorphism on cohomology:
  \[ \Hoh_G^*(W\times U_{r+t}) \iso \Hoh^*((W\times U_{r+t})/G) \overset{f^*}{\longrightarrow} \Hoh^*(\Spec R) \iso
    \Hoh^*_G(T). \]
 We now define the $R$-algebra $B = {\, }^T A$ by twisting $A$ by $T$. We claim that $B_\C$ cannot be generated by $r$
  elements. 
  
  Suppose for the sake of contradiction that $B_\C$ can be generated by $r$ elements. Then there is a $G_\C$-equivariant map 
  $\phi \colon T_\C \to U_{r, \C}$ representing $B_\C$ and a generating $r$-tuple of elements, 
  by Proposition~\ref{prop.equivariant-map-corres}.  
  This leads us to the following commutative diagram 
  \begin{equation} \label{eq:4tri}
    \xymatrix{  & \ar_\iso[dl] \ar[d] \Hoh^i(BG) \ar[d] \ar[dr] \ar@/^1em/[drr]^{\quad \text{non-injective}} \\  \Hoh^i_G(U_{r+t})   \ar_-\iso[r] & \Hoh^i_G(W \times U_{r+t})  \ar@{->}[r]_-\iso & \Hoh^i_G(T) & \Hoh^i_G(U_r) \ar@{->}[l]_{\phi^*} }
  \end{equation}
  in which the first three groups in the bottom row are isomorphic.  The rightmost triangle in \eqref{eq:4tri} is the
  triangle of Lemma \ref{lem:prohibitGeneration}, and the leftmost diagonal arrow is an isomorphism by Lemma
  \ref{lem:isoRange}. The isomorphism $\Hoh^i(BG) \to \Hoh^i_G(T)$ in \eqref{eq:4tri} must factor
  through the non-injective map $\Hoh^i(BG) \to \Hoh^i_G(U_r)$, a contradiction. Therefore, $B_\C$ cannot be generated by $r$
  elements. This proves the claim.
 
  Unfortunately, $B = {\, }^T A$ is not yet the algebra we are looking for. Indeed, 
  \[ \Kdim(R) \geqslant \dim \, (W \times U_{r + t})/ G > \dim(U_{r + t}/G) = (r + t)\dim(A) - \dim(G) \geqslant i ; \]
  see~\eqref{e.choice-of-t}, whereas we want $\Kdim(R) = i$. To complete the proof, we modify the torsor
  $T \to \Spec(R)$ to reduce $\Kdim(R)$ to $i$. To do this, we embed $\Spec R $ into an affine space $\mathbb A^N$. By
  Theorem~\ref{th:affineLefschetz}, provided $\Kdim(R) \ge i + 1$, there is a smooth affine hyperplane section
  $\iota \colon \Spec R' \hookrightarrow \Spec R$ such that $\iota^*:\Hoh^i(\Spec R) \to \Hoh^i(\Spec R')$ is
  injective. If we write $T'$ for the pullback of the $G$-torsor $T$ along $\iota$, we deduce that
  $\iota^* : \Hoh_G^i(T) \to \Hoh_G^i(T')$ is also injective. Observe that $\Kdim R' = \Kdim R -1$.

  The $R'$-algebra $B' = ({\, }^{T'} A)_{\CC}$ cannot be generated by $r$ elements, otherwise there would exist a
  $G$-equivariant morphism $\phi \colon T'_\C \to U_{r, \C}$ giving rise to a commutative diagram:
\[
    \xymatrix{   \Hoh^i(BG) \ar_{\iso}[dr] \ar@/^1em/[drr] & & \\   & \Hoh^i_G(T) \ar@{^{(}->}[d] & \Hoh^i_G(U_r)
      \ar@{->}[dl]^{\phi^*} \ar[l] \\
                                                                      & \Hoh^i_G(T').  & }
  \]  
  Here the injective map $\Hoh^i(BG) \to \Hoh^i_G(T')$ factors
  through the non-injective map $\Hoh^i(BG) \to \Hoh^i_G(U_r)$, which is a contradiction.
  Replacing $R$ by $R'$ and $T$ by $T'$,
  we reduce $\Kdim  R$ by $1$ while preserving the property that $\gen_{R_{\mathbb C}}(B_{\mathbb C}) > r$.
 Since this procedure can be repeated as long as $\Kdim(R) > i$, we eventually arrive at an example where $\Kdim   R = i$. 
\end{proof}

\section{Algebras requiring many generators}
\label{sect.main3}

Recall that if $X$ is defined over $k \hookrightarrow \CC$, then we write $\Hoh^*(X)$ for $\Hoh^*(X(\CC))$. Similarly, we
write $\Hoh^*(BG)$ in place of $\Hoh^*(BG(\CC))$.

In this section, we establish Theorem \ref{thm.main3} by applying Theorem \ref{thm:produceCounterexample}. We must first
prove that $\Hoh^i(BG)$ is nonzero for many values of $i$, which we do in Lemma \ref{lem:cohoEvery4}. Our main tool,
singular cohomology of the complex points, can be used only when the base field is embedded in $\CC$, but by recourse to
the Lefschetz principle in
Lemma~\ref{lem:requireForLefschetz}, we will be able prove Theorem \ref{thm.main3} over any field of characteristic $0$.

\begin{lemma} \label{lem:cohoEvery4}
  Let $k$ be a field with a fixed embedding in $\C$. Let $G$ be an affine algebraic group over $k$ that is not
  unipotent. Then there exists a natural number $\rho_G$ such that for all
  $i \ge 1$, there exists $j \in \{i+1, i+2,\dots, i+\rho_G\}$ such that
  \[ \Hoh^j(BG) \not \iso 0. \]
  If $H$ denotes the quotient of the identity connected component of $G$ by its unipotent radical, and $Z$ is the centre of $H$, then we can take $\rho_G=4$ if $\dim H>0$, and $\rho_G=2$ if $\dim Z>0$.
\end{lemma}
\begin{proof}
  Since the cohomology depends only on $G(\CC) = G_\CC(\CC)$, and since $G_\CC$ is unipotent if and only if $G$ is, there is nothing
  to be lost by assuming $k=\CC$.
  
  The proof proceeds by a sequence of reductions. First, consider the natural exact sequence
  \begin{equation*}
    \begin{tikzcd}
      1 \rar & N \rar & G \rar & H \rar & 1 
    \end{tikzcd}
  \end{equation*}
  of algebraic groups, where $N$ is the unipotent radical and $H$ is reductive; see~\cite[Theorem 4.3]{Hochschild1981}. 
  Since $N$ is unipotent, it admits a composition series by subgroups having quotients isomorphic to $\Ga$ \cite[IV Th.~de Lazard
  4.1]{Demazure1970}. It follows that $N$ is contractible as a topological space and so $BG \weq BH$. Without
  loss of generality, therefore, we may assume that $G$ is reductive.

   If $G$ is finite, then $\rho_G$ exists by \cite[Theorem 2.4]{Benson1990}. Thus, we may assume that
  $G$ is infinite, i.e., $\dim G > 0$. We will show that in this case $ \Hoh^j(BG; \QQ) \not \iso 0$ (and hence, $\Hoh^j(BG; \ZZ) \not \iso 0$)
  for some 
  $j \in \{i+1, i+2, i+3, i+4\}$. 
  Let $G^0$ denote the identity component of $G$.
  Since the component group $G/G^0$ is finite,
  $\Hoh^j(BG^0; \QQ) = \Hoh^j(BG; \QQ)$.  
  We may therefore assume that $G$ is nontrivial, connected and reductive.
  Let $Z$ be the centre of $G$. We will treat the cases where $Z$ is finite and infinite separately.

  If $Z$ is finite, then $G$ is a semisimple linear algebraic group. By \cite[Theorem 21.2]{Chevalley1948}, we know that
  $\tilde \Hoh^l(G; \QQ) \iso 0$ for values of $l \in \{0,1,2\}$ but $\tilde \Hoh^3(G; \QQ) \not \iso 0$. It is an easy
  consequence of the Serre spectral sequence that $\tilde \Hoh^l(BG; \QQ) \iso 0$ for $l \in \{0,1,2,3\}$ but there
  exists a nonzero element $\alpha \in \tilde \Hoh^4(BG; \QQ)$. We know, however, that
  $\Hoh^*(BG; \QQ) \iso \Hoh^*(BT; \QQ)^W$, where the latter denotes the Weyl invariants of the cohomology of a maximal
  torus of a maximal compact subgroup of $G(\C)$, by \cite[Theorem 11]{Malcev1945} and \cite[Reduction 2,
  p36]{Hsiang1975}. Since $\Hoh^*(BT; \QQ)$ is a polynomial ring in some number of generators, the element $\alpha$ of
  the subring $ \Hoh^*(BT; \QQ)^W$ cannot be nilpotent. Therefore the elements $\alpha^l$ form a family of nonzero
  elements in $\Hoh^{4l}(BG; \QQ)$. In this case, we may take $\rho_G = 4$.

  If $Z$ is not finite, then $G$ contains a nontrivial torus $Z'$ such that the quotient $G^{\text{ss}} = G/Z'$ is semisimple. In
  particular, $\tilde \Hoh^l(BG^{\text{ss}}; \QQ)$ vanishes for values of $l \in \{0, 1, 2, 3\}$. A Serre spectral sequence
  argument applied to
  \[
    \begin{tikzcd}
      BZ' \rar & BG \rar & BG^{\text{ss}}
    \end{tikzcd}
 \]
  shows that we can find a nontrivial element $\alpha \in \Hoh^2(BG; \QQ)$ in this case, and we can repeat the
  same argument as in the previous case to show that the $\alpha^l$ form a family of nonzero elements in $\Hoh^{2l}(BG;
  \QQ)$. In this case, we may take $\rho_G = 2$.
\end{proof}

\begin{remark}
 The construction of $\rho_G$ in Lemma \ref{lem:cohoEvery4} uses the algebraic group structure of $G_\CC$. However, since
 $\Hoh^*(BG(\CC))$ depends only on the Lie group $G(\CC)$, we may assume that $\rho_G$ that depends only on $G(\CC)$ and not on $G$.
\end{remark}

\begin{proof}[Proof of Theorem~\ref{thm.main3}]
   First of all, we consider the case where the field $k$ is a subfield of $\CC$. Set $r$ to be the largest integer satisfying
  \begin{equation} \label{e.pf-of-main3b}
  d \geqslant 2rn - 2 \dim_{\CC} G + \rho_G .
  \end{equation}   
In other words, set 
  \begin{equation} \label{e.pf-of-main3a}  r = \Floor{\dfrac{d + 2 \dim G - \rho_G}{2n}}. \end{equation}
By Lemma~\ref{lem:cohoEvery4}, there exists an integer $d + 1 - \rho_G \leqslant i \leqslant d$ such that $\Hoh^i(BG) \neq 0$. By~\eqref{e.pf-of-main3b},
  $i > 2rn - 2\dim_{\CC} G$. 
  
  We claim that $\Hoh_{G}^i(U_r) = 0$.  Consider the quotient $U_r(\C) \to U_r(\C)/G(\C)$. This is a principal $G(\C)$-bundle over a manifold of dimension $2rn-2\dim_\C G$, because the Lie group $G(\C)$ acts properly on $U_r(\C)$ by Lemma \ref{pr:schemeTheoreticallyFree}. The quotient map $EG(\C) \times_{G(\C)} U_r(\C) \to U_r(\C)/G(\C)$ is a fibre bundle with fibre $EG(\C)$, as one can see by using an open cover of the manifold $U_r(\C)/G(\C)$ trivializing
$U_r(\C) \to U_r(\C)/G(\C)$. This implies that $\Hoh^\ast_{G}(U_r)\iso \Hoh^\ast (U_r(\CC)/G(\CC))$,
see~\cite[Theorem~7.57]{James1984}. On the other hand, we know that $\Hoh^j(U_r(\CC)/G(\CC)) = 0$ whenever $j >  2rn -2 \dim G$, since
$U_r(\C)/G(\C)$ is a manifold of real dimension $2rn -2 \dim_{\CC} G$. This proves the claim.
  
  We conclude that the map on equivariant cohomology
   $\Hoh^i(BG) \to \Hoh_{G}^i(U_r)$ is not injective. By~Theorem \ref{thm:produceCounterexample}, there exists a $k$-ring $R$ of finite type and
   an $R$-form $B$ of $A$ such that $\dim R = i \leqslant d$ and $\gen_{R_\CC}(B_{\CC}) > r$. Substituting in the formula for $r$ 
   from~\eqref{e.pf-of-main3a} and remembering $\gen_{R_{\CC}}(B_{\CC})$ is an integer,
   we can rewrite this inequality as
   \[ \gen_{R_{\CC}}(B_{\CC}) > \dfrac{d + 2 \dim G - \rho_G}{2n} . \]
  If $i = \Kdim R < d$, we replace $R$ by the polynomial ring $R' = R[t_1, \ldots, t_{d - i}]$ and $B$ by 
   $B' = B \otimes_R R'$; it is easy to see that $\Kdim R' = d$ and 
   \[ \gen_{R'_{\CC}}(B'_{\CC}) = \gen_{R_{\CC}}(B_{\CC}) > \frac{d + 2 \dim G - \rho_G}{2n}. \]
   By Lemma~\ref{lem:requireForLefschetz}\ref{le1}, this
   completes the proof of Theorem~\ref{thm.main3} in the case where $k$ is a subfield of $\CC$.
   
  Now we establish the general case, where $k$ is arbitrary field of characteristic $0$.
  Since the finite-dimensional $k$-algebra $A$ may be defined by means of finitely many structure constants,
  there exists a subfield $k_0 \subset k$ and a finite-dimensional $k_0$-algebra $A_0$, 
  such that $k_0$ is finitely generated over $\QQ$ and $A_0 \tensor_{k_0} k \iso A$. Since $k_0$ is finitely generated
  over $\QQ$, we may embed $k_0$ in $\CC$. Using the first part of this proof, we produce a $k_0$-ring $R_0$ and a
  form $B_0$ of $A_0$ over $R_0$ such that $\Kdim(R_0) = d$ and 
  \[ \gen_{(R_0)_\CC}((B_0)_{\CC}) > \frac{d + 2 \dim G - \rho_G}{2n}. \]
  By Lemma~\ref{lem:requireForLefschetz}\ref{le3}, 
  \[ \gen_{R}(B) > \frac{d + 2 \dim G - \rho_G}{2n} , \]
where $R = R_0 \otimes_{k_0} k$ is a finite type $k$-ring and $B = B_0 \otimes_{k_0} k$ is an $R$-form of $A$.
Moreover, $\Kdim R = \Kdim R_0 = d$, 
as required.
\end{proof}  

\begin{remark} \label{rem.unipotent} Theorem~\ref{thm.main3} fails if 
$G$ is unipotent. Indeed, in this case every $G$-torsor $T \to \Spec R$ splits, see 
e.g.,~\cite[Corollary 3.2]{Asok2007}.\footnote{Here we are assuming that $\operatorname{char}(k) = 0$, as in Theorem~\ref{thm.main3}. 
This is also a standing assumption in~\cite{Asok2007}.} 
This tells us that the only $R$-form $B$ of $A$ (up to isomorphism) is $B = A \otimes_k R$ and consequently,
$\gen_R(B) \leqslant \gen_k(A)$ for every $R$.
\end{remark}

\begin{remark} \label{rem.aut-arbitrary}
  Even among the class of finite-dimensional commutative associative and unital algebras it is possible for the
  automorphism group to be nontrivial and unipotent. In
  \cite[Example 1.7]{Pollack1989}, R.~D.~Pollack shows that the automorphism group of 
  the finite-dimensional commutative, associative and unital $\CC$-algebra
  \[ A = \frac{\CC[x,y]}{(x,y)^5 + (x^2-y^3, x^3-y^4)} \]
    is unipotent. Note that this group is nontrivial,
    since $x \mapsto x + y^4,\, y\mapsto y$ gives a nontrivial automorphism. 
%
  \end{remark}

\section{Geometry of matrix tuples}
\label{sect.matrix-tuples}

In this section and the next, we return to the study of the specific case where $A$ denotes the algebra of $s \times s$
matrices, $\Mat_s(k)$. We will assume $k$ is  field of characteristic $0$ and $s \geqslant 3$ throughout. 
We continue to use the notation $V_r$ to denote the affine variety whose
$k$-points are $A^r$, the notation $Z_r$ to denote the closed subscheme of $r$-tuples that do not generate $A^r$, and
the notation $U_r$ for $V_r \sm Z_r$. Recall from Section~\ref{subsect.azumaya} that
\[ Z_r^{\mathrm{red}} = X_1 \cup \ldots \cup X_{s-1}, \] where $X_i$ is the variety 
of $r$-tuples $(a_1, \dots, a_r)$ of $s \times s$-matrices having a common $i$-dimensional invariant
subspace. Proposition~\ref{prop.Z}(a) says that $X_i$ is a closed irreducible subvariety of $V_r$ of dimension
$rs^2 - (r-1)i(s-i)$. We will be particularly interested in $X_1$, along with the $\PGL_s$-invariant subvariety $T_2$ of $V_r$ 
that we define as the variety of $r$-tuples $(a_1, \dots, a_r)$ of $s\times s$-matrices satisfying the following conditions:

\begin{itemize}
\item there exists a subspace $W$ having dimension at least $2$, invariant under each $a_i$, $i = 1, \ldots, r$;
\item $a_1, \ldots, a_r$ commute pairwise when restricted to $W$.
\end{itemize}

\smallskip
Our goal is to prove Theorem~\ref{thm.azumaya}(b). We will do this in the next section by showing that the natural map
\[ \Hoh^{2(r-1)(s-1)}( B\PGL_s) \to \Hoh^{2(r-1)(s-1)}_{\PGL_s}(U_r) \] is not injective, then appealing to Theorem~\ref{thm:produceCounterexample}. 
Since cohomological tools apply best to smooth varieties and transverse intersections, our focus in this section is proving
the following two preparatory results.

\begin{proposition} \label{prop.Zsmooth2}
The variety $X_1$ is smooth away from $T_2$.
\end{proposition}

\begin{proposition}\label{prop.XcapY}
Fix an invertible $(s-1) \times (s-1)$ matrix $a$ such that none of the standard basis vectors
of $k^{s-1}$ is an eigenvector for $a$. Fix
$(s-1)$ distinct elements $\lambda_2, \dots, \lambda_s \in k^\times$.
Let $Y$ be the affine subspace of  $V_r = \Mat_s^r$  consisting of $r$-tuples of the form 
\[ \left(\begin{bmatrix}
 0 & 0 & \dots & 0 \\ 0 & & & \\ \vdots & & a & \\ 0 & & &
\end{bmatrix},
\begin{bmatrix}
 0 & 0 & \dots & 0 \\ x_{22} & \lambda_2 & & 0 \\ \vdots & & \ddots & \\ x_{2s} & 0 & & \lambda_s
\end{bmatrix},
\begin{bmatrix}
 0 & 0 & \dots & 0 \\ x_{32} & \lambda_2 & & 0 \\ \vdots & & \ddots & \\ x_{3s} & 0 & & \lambda_s
\end{bmatrix}, \dots, \begin{bmatrix}
 0 & 0 & \dots & 0 \\ x_{r2} & \lambda_2 & & 0 \\ \vdots & & \ddots & \\ x_{rs} & 0 & & \lambda_s
\end{bmatrix} \right).
\]
Then $Y \cap T_2 = \emptyset$, and $Y$ intersects $X_1$ transversely in $V_r = \Mat_s^r$.
\end{proposition}

Note that in the definition of $Y$ the matrix $a$ and the distinct nonzero scalars $\lambda_2, \ldots, \lambda_s$ remain fixed, 
and each of $x_{22}, x_{23}, \ldots, x_{rs}$ varies over $k$. 
Thus $Y$ is isomorphic to the affine space $\A_k^{(s-1)(r-1)}$
linearly embedded into $V_r$. Here and subsequently
\[ \vec{e}_1 = \begin{bmatrix} 1 \\ 0 \\ \vdots \\ 0 \end{bmatrix}, \quad  \ldots, \quad \vec{e}_s
= \begin{bmatrix} 0 \\  \vdots \\ 0 \\ 1 \end{bmatrix}  \]
denote the standard basis vectors of $k^s$. 

\begin{remark} \label{rem:pXcapYneeds3}
  In the case of $s=2$, the statement of Proposition \ref{prop.XcapY} does not make sense, since there is no $1 \times
  1$ matrix $a$ such that none of the standard basis vectors of $k^1$ is an eigenvector of $a$. This is why we restrict
  to $s \ge 3$ in this section.
\end{remark}

Since Propositions~\ref{prop.Zsmooth2} and~\ref{prop.XcapY}
are geometric in nature, we will assume that $k$ is an algebraically closed field of characteristic $0$ for the remainder of this section.

\subsection*{Proof of Proposition~\ref{prop.Zsmooth2}}
Let $k \langle x_1, \dots, x_r \rangle$ denote the free unital associative $k$-algebra in $r$ non-commuting variables $x_1, \ldots, x_r$.
Let $C$ be the $2$-sided ideal generated by the commutators $[x_i, x_j]$.

Let $E = V_r \times k^s$ be the trivial vector bundle of rank $s$ over $V_r$. For every $f \in k \langle x_1, \ldots, x_r \rangle$, we define
an endomorphism $L_f \colon E \to E$ by $L_f(v) = f(a_1, \ldots, a_r) \cdot v$ over $(a_1, \ldots, a_r) \in V_r$. Let
\[ N = \bigcap_{f \in C} \, \Ker(L_f) \] 
be the intersection of the kernels of $L_f$, where $f$ ranges over $C$. It follows from the noetherian property of $E$ that
$N$ is, in fact, the kernel of the map 
\[ L = L_{f_1} \times \ldots \times L_{f_N} \colon E \to E \times_{V_r} E \times_{V_r} \ldots \times_{V_r} E \]
of vector bundles over $V_r$ for finitely many elements $f_1, \ldots, f_N \in C$. Let $(V_r)_{\nullity \geqslant i}$ (respectively, $(V_r)_{\nullity = i}$)
be the closed (respectively, locally closed) subvariety of $V_r$ where the nullity of $L$ is $\leqslant i$ (respectively, $= i$).
Then $N$ is a vector bundle of rank $i$ over $(V_r)_{\nullity = i}$. 

\begin{lemma} \label{lem.invariant}
Denote the fibre of $N$ over $a = (a_1, \ldots, a_r) \in V_r$ by $N_a$. Then 

\smallskip
(a) $N_a$ is invariant under $a_i$ for each $i = 1, \ldots, r$.

\smallskip
(b) $(V_r)_{\nullity \geqslant 1} = X_1$ and $(V_r)_{\nullity \geqslant 2} = T_2$.
\end{lemma}

\begin{proof}
(a) Suppose $f \in C$. Then $f x_i \in C$ for every $i = 1, \ldots, r$. Consequently, 
if $v$ lies in $N_a$, then $L_f(a_i v) = L_{f x_i}(v) = 0$ for every $f \in C$. Thus $a_i v$ also lies in $N_a$.
 
(b) For the first equality, note that if a $1$-dimensional subspace $V$ of $k^s$ is invariant under $a_i$ for each $i = 1, \ldots, r$, then clearly the restrictions of $a_1, \ldots, a_r$
to $V$ commute pairwise. This shows that $X_1 \subset (V_r)_{\nullity \geqslant 1}$. To prove the opposite inclusion, suppose $a \in (V_r)_{\nullity \geqslant 1}$,
Then $\dim N_a \geqslant 1$ and the restrictions of $a_1, \ldots, a_r$ to $N_a$ commute pairwise. Since we are assuming that our base field $k$ is algebraically closed, $a_1, \ldots, a_r$ have a common eigenvector in $N_a$. Hence, $a \in X_1$. 

The second equality, $(V_r)_{\nullity \geqslant 2} = T_2$, follows directly from the definition of $T_2$.
\end{proof}

\begin{remark} \label{rm:codimT}
 It is easy to see that $T_2$ is, in fact, a proper closed subvariety of $X_2$; in particular, $\dim T_2 < \dim X_2$. 
\end{remark}

We are now ready to finish the proof of Proposition~\ref{prop.Zsmooth2}.
Let $N_1$ be the restriction of $N$ to 
\[ (V_r)_{\nullity  = 1} = (V_r)_{\nullity \geqslant 1} - (V_r)_{\nullity \geqslant 2} = X_1 - T_2. \]
Then $N_1$ is a vector bundle of rank $1$ over $(X_1 - T_2)$. Its projectivization $\mathbb{P}(N_1)$
is the incidence variety in $(X_1 - T_2) \times \PP^{s-1}$ consisting of pairs
\[ ((a_1, \dots, a_r), L), \]
where $L$ is a $1$-dimensional subspace of $k^s$ invariant under $(a_1, \ldots, a_r) \in (X_1 - T_2)$.
Since $N_1$ is a vector bundle of rank $1$ over $(X_1 - T_2)$, the natural projection $\mathbb{P}(N_1) \to (X_1 - T_2)$ is an isomorphism.

 It thus suffices to prove that $\bbP(N_1)$ is smooth. To prove that $\PP(N_1)$ is smooth, we will take advantage of the fact that the natural
 projection map $N \to \PP^{s-1}$ given by $\pi \colon ((a_1, \ldots, a_r), L) \mapsto L$ is
 $\PGL_s$-equivariant, 
 and $\PGL_s$ acts transitively on the target $\PP^{s-1}$. This tells us that $\pi$ is flat and the fibres of $\pi$ 
 over the $k$-points of $\PP^{s-1}$ are pairwise isomorphic.  
 Moreover, the fibre over any closed point $v \in \PP^{s-1}$ is an
 open subvariety of the affine space of all $r$-tuples of matrices having $v$ as an eigenspace. 
 Hence, it is smooth. Therefore, $\pi \colon \PP(N_1) \to \PP^{s-1}$ 
 is flat with smooth fibres. We conclude that $\pi$ is a smooth morphism, and consequently,
 $\PP(N_1)$ is smooth over $\Spec k$, as claimed.
\qed

\subsection*{Proof of Proposition~\ref{prop.XcapY}}
\begin{lemma} \label{lem.YmeetX1}
The scheme-theoretic intersection of the subvarieties $Y$ and $X_1$ of $V_r$ is a single reduced point
\[ p := \left(\begin{bmatrix}
 0& 0 & \dots & 0 \\ 0 & & & \\ \vdots & & a & \\ 0 & & &
\end{bmatrix},
\begin{bmatrix}
 0 & 0 & \dots & 0 \\ 0& \lambda_2 & & \\ \vdots & & \ddots & \\ 0 & & & \lambda_s
\end{bmatrix},
\begin{bmatrix}
 0 & 0 & \dots & 0 \\ 0 & \lambda_2 & & \\ \vdots & & \ddots & \\ 0 & & & \lambda_s
\end{bmatrix}, \dots, \begin{bmatrix}
 0 & 0 & \dots & 0 \\ 0 & \lambda_2 & & \\ \vdots & & \ddots & \\ 0 & & & \lambda_s
\end{bmatrix} \right).\]
\end{lemma}
\begin{proof}
We claim that $p$ is the only geometric point in the intersection $X_1 \cap Y$. To see this, let $L$ be an algebraically
closed field and let $(a_1, \dots, a_r) \in (X_1 \cap Y)(L)$, where
\[ a_1 = \begin{bmatrix}
 0 & 0 & \dots & 0 \\ 0 & & & \\ \vdots & & a & \\ 0 & & &
\end{bmatrix} \quad \text{and} \quad
 a_i =
\begin{bmatrix}
 0 & 0 & \dots & 0 \\ x_{i2} & \lambda_2 & & 0 \\ \vdots & & \ddots & \\ x_{is} & 0 & & \lambda_s
\end{bmatrix} \]
for $i = 2, \ldots, r$. Here $x_{ij}$ are elements of $L$.
The eigenvectors of $a_1$
are nonzero multiples of $\vec{e}_1$ and $\begin{bmatrix} 0 \\ v_2 \\ \vdots \\ v_s \end{bmatrix}$ ,
where $\begin{bmatrix} v_2 \\ \vdots \\ v_s  \end{bmatrix}$ is an eigenvector of the invertible matrix $a$.
The eigenvectors of $a_i$ for $i \ge 2$ are nonzero multiples of $\vec{e}_2, \ldots, \vec{e}_s$ 
and of $\begin{bmatrix} 1 \\ -\lambda_2^{-1}x_{i2} \\ \vdots \\
-\lambda_s^{-1}x_{is} \end{bmatrix}$. We conclude that
the matrices $a_1, a_2, \dots a_r$ do not have a common eigenvector unless $x_{22} = x_{23} = \dots = x_{ns} = 0$.
In the case where $x_{22} = x_{23} = \ldots = x_{ns} = 0$, the only
eigenspace shared by $a_1, \ldots, a_r$ is $\Span( \vec{e}_1 )$. 
This proves the claim. 

It remains to show that the scheme-theoretic intersection $X_1\cap Y$ is reduced, i.e., that it is isomorphic to $\Spec k$. 
Let $(R, \mathfrak m)$ denote a local ring containing $k$. We wish to show that there is a unique morphism $f: \Spec R \to X_1 \cap Y$ 
given by composing $p \colon \Spec k \to X_1 \cap Y$ with the natural projection $\Spec R \to \Spec k$.
Write $L$ for the algebraic closure of $R/\mathfrak m$. As we showed above, the composite $\Spec L \to \Spec R \to X_1 \cap Y$
must be the geometric point $p$. Therefore, $f: \Spec R \to X_1 \cap Y$ must correspond to an $R$-valued point of the form
\[ q=\left(\begin{bmatrix}
 0 & 0 & \dots & 0 \\ 0 & & & \\ \vdots & & a & \\ 0 & & &
\end{bmatrix},
\begin{bmatrix}
 0 & 0 & \dots & 0 \\ x_{22} & \lambda_2 & & 0 \\ \vdots & & \ddots & \\ x_{2s} & 0 & & \lambda_s
\end{bmatrix},
\ldots, 
\begin{bmatrix}
 0 & 0 & \dots & 0 \\ x_{r2} & \lambda_2 & & 0 \\ \vdots & & \ddots & \\ x_{rs} & 0 & & \lambda_s
\end{bmatrix} \right)\]
where each $x_{ij}$ lies in $\mathfrak m$.
The condition that $q$ lies in $X_1(R)$ means that
there exists a unimodular vector $\vec v := \begin{bmatrix} v_1 \\ \vdots \\ v_s \end{bmatrix} \in R^s$ such that 
$a_1 \vec v= \mu_1 \vec v,  \ldots, a_r \vec{v} = \mu_r \vec{v}$
for some $\mu_1, \ldots, \mu_r \in R$. If we reduce modulo $\mathfrak m$, then as we noted in the previous paragraph, 
the only eigenspace shared by $a_1, \ldots, a_r$ is $\Span(\vec{e}_1)$. In other words, if we reduce modulo $\mathfrak m$,
the vector $\vec v$ will become a scalar multiple of $\vec{e}_1$. 
Therefore $v_1 \in R^\times$, while $v_2, \dots, v_s \in \mathfrak m$. The condition that $\vec v$ is an
eigenvector of $a_1$ gives
\[  \mu_1  \begin{bmatrix} v_1 \\ v_2 \\ \vdots \\ v_s \end{bmatrix} =  \begin{bmatrix}
 0 & 0 & \dots & 0 \\ 0 & & & \\ \vdots & & a & \\ 0 & & &
\end{bmatrix} \begin{bmatrix} v_1 \\ v_2 \\ \vdots \\ v_s \end{bmatrix} = \begin{bmatrix} 0 \\ w_2 \\ \vdots  \\ w_s \end{bmatrix},
\text{ where } \begin{bmatrix}  w_2 \\ \vdots  \\ w_s \end{bmatrix}
= a \begin{bmatrix} v_2 \\ \vdots\\  v_s\end{bmatrix}.\]
Because $v_1$ is a unit, $\mu_1 = 0$, so that $\vec w = 0$. Since $a$ is an invertible $(s-1) \times (s-1)$-matrix, it follows that $v_2 = \dots =
v_s = 0$. That is, we have deduced that  $\vec{e}_1$ is a unit multiple of $\vec v$.

We may now suppose $\vec{e}_1$ is an eigenvalue of each matrix. As in the field case, looking at the first entry of $a_i \vec e_1 = \lambda_i
\vec e_1$ for
$i=1,\dots, r$ implies that the associated eigenvalues are $0$. Therefore $x_{ij} = 0$ for all $i = 2, \ldots, r$ and $j = 2, \ldots, s$.
This proves that $q$ is uniquely determined, so that $X_1 \cap Y$ is reduced, as claimed.
\end{proof}

\smallskip
We are now ready to complete the proof of Proposition~\ref{prop.XcapY}. To show that $Y \cap T_2 = \emptyset$, observe that 
$T_2 \subset X_1$; see Lemma~\ref{lem.invariant}(b). By Lemma~\ref{lem.YmeetX1}, $Y \cap X_1$ is a single point $p = (a_1, a_2, \ldots, a_r)$,
where 
\[ a_1 = \begin{bmatrix}
 0& 0 & \dots & 0 \\ 0 & & & \\ \vdots & & a & \\ 0 & & &
\end{bmatrix} \quad \text{and} \quad
a_2 = \ldots = a_r = \begin{bmatrix}
 0 & 0 & \dots & 0 \\ 0& \lambda_2 & & \\ \vdots & & \ddots & \\ 0 & & & \lambda_s
\end{bmatrix}. \]
Hence, it suffices to show that
$p \not \in T_2$.  Assume the contrary: there exists a subspace $V \subset k^s$ of dimension $\geqslant 2$ which is invariant under
$a_1$ and $a_2$ and such that 
$a_1$ and $a_2$ commute when restricted to $V$. Note that $\Span_k(\vec{e}_2, \ldots, \vec{e}_s)$ is invariant under both $a_1$ and $a_2$.
Hence so is $V' =  V \cap \Span_k(\vec{e}_2, \ldots, \vec{e}_n)$. Since $\dim V \geqslant 2$, we conclude that $\dim V' \geqslant 1$.
Since $a_1|_{V'}$ and $a_2|_{V'}$ commute, $a_1$ and $a_2$ must therefore share an eigenvector in $V'$. But the only eigenvectors
$a_1$ and $a_2$ share in $k^s$ are scalar multiples of $\vec{e}_1$, a contradiction. We conclude that $p \not \in T_2$.

It remains to show that $X_1$ and $Y$ intersect transversely at $p$. Since $p \not \in T_2$, the variety $X_1$ is smooth at $p$ by
Proposition~\ref{prop.Zsmooth2}. Since $Y$ is an affine space, $Y$ is also smooth at $p$. By definition $Y$ is of dimension
$(s-1)(r-1)$, and by Proposition~\ref{prop.Z}(a), $X_1$ is of dimension $rn^2 - (s-1)(r-1)$.
Their intersection is a single reduced point $p$, by Lemma~\ref{lem.YmeetX1}. Since 
  \[ \dim T_p (V_r) = rn^2 = \dim T_p (X_1) +  \dim T_p (Y), \]
$X_1$ and $Y$ intersect transversely in $V_r$.
\qed

\section{Proof of Theorem~\ref{thm.azumaya}(b)}
\label{sect.azumaya(b)}

In this section, unless otherwise stated, $k$ denotes a subfield of $\CC$. We will show that the natural map
\[ \Hoh^{2(s-1)(r-1)}(B\PGL_s) \to \Hoh^{2(s-1)(r-1)}_{\PGL_s}(U_r)\]
is not injective.

Consider the $1$-parameter subgroup $\Gm \hookrightarrow \PGL_s$ given by $\lambda \mapsto \begin{bmatrix} \lambda^{-1} & 0 \\ 0 & I_{s-1} \end{bmatrix}$.
The group $\Gm$ acts on $\Mat_s$ by 
 \[ \lambda \cdot \begin{bmatrix}
  a_{11} & a_{12} & \dots & a_{1n} \\ a_{21} & a_{22} & \dots & a_{2n} \\ \vdots & \vdots & \ddots & \vdots \\ a_{n1} &
  a_{n2} & \dots & a_{nn}
 \end{bmatrix} = \begin{bmatrix}
  a_{11} & \lambda^{-1} a_{12} & \dots & \lambda^{-1} a_{1n} \\ \lambda a_{21} & a_{22} & \dots & a_{2n} \\ \vdots & \vdots & \ddots & \vdots \\ \lambda a_{n1} &
  a_{n2} & \dots & a_{nn}
 \end{bmatrix}\]
 Recall that the variety $Y$ from Proposition \ref{prop.XcapY} is isomorphic to $\A^{(s-1)(r-1)}$, the affine space with
 linear coordinates $(x_{22}, \dots, x_{rn})$. The variety $Y$ is $\Gm$-invariant as a subvariety of $V_r$, and the isomorphism $Y \iso
 \A^{(s-1)(r-1)}$ is $\Gm$-equivariant, where $\Gm$ acts on the linear coordinates by
 \begin{equation}
   \label{eq:1}
    \lambda \colon (x_{22}, \dots, x_{rn}) \mapsto (\lambda x_{22}, \dots, \lambda x_{rn}).
 \end{equation}

 The normal bundle $N$ of ${\Spec k} \iso X_1 \cap Y$ in
 $Y \iso \A^{(s-1)(r-1)}$ is trivial of rank $(s-1)(r-1)$, and the total space may be identified with $Y$ itself. From
 \eqref{eq:1}, $N$ is isomorphic to the direct sum of $(s-1)(r-1)$ copies of the standard representation of $\Gm$ on
 $\A^1$. The top Chern class of this representation, $c_{(s-1)(r-1)}^{\Gm}(N)$, is 
 $\theta^{(s-1)(r-1)} \in \Hoh^{2(s-1)(r-1)}_{\Gm}({\Spec k})$,
 where $\theta=c^{\Gm}_1(\mathbb{A}^1)$ is a generator of
 $\Hoh^2_{\Gm}(\Spec k)=\Hoh^2(B \Gm)\cong \ZZ$.

 \begin{remark} \label{not:identifications}
   We describe four identifications, which we use for the rest of the section.

   First, because $X_1 \cap Y = {\Spec k}$, we may identify $\Hoh^*_\Gm(X_1 \cap Y)$ with $\ZZ [\theta]$, where $|\theta|
   =2$.  \benw{There is a question of $\theta$ versus $-\theta$. Is it necessary to specify which generator $\theta$ is?}

  Second, we use the isomorphism $i^*: \Hoh_\Gm^*(Y) \to \Hoh_\Gm^*(X_1 \cap Y)$ to identify $\Hoh_\Gm^*(Y)$ with $\ZZ[\theta]$ as well.

  Third, because of the isomorphism $V_r \iso \A_k^{rs^2}$, we may also identify $\Hoh_\Gm^*(V_r)$ with $\ZZ[\theta]$.

  Finally, recall from Proposition~\ref{prop.Z}(a) that $X_2$ is a closed irreducible subvariety of $V_r$ of codimension $2(s-2)(r-1)$.
  By Remark~\ref{rm:codimT}, the codimension of $T_2$ in $V_r$ is at least $2(s-2)(r-1)+1$. We deduce (using e.g.,
  Lemma \ref{lem:localizationIsoEquiv}) that for values of $j< 4(s-2)(r-1)+1$, the natural inclusion $(V_r - T_2) \hookrightarrow V_r$ 
  induces an isomorphism
  \[ \Hoh_\Gm^{j}(V_r - T_2 ) \isomto \Hoh_\Gm^{j}(V_r).\]
  Using this, we identify $\Hoh_\Gm^j(V_r - T_2)$ with $\Hoh_\Gm^j(V_r)=\Hoh^j(B\Gm)$ for every $j \leqslant 4(s-2)(r-1)$.
\end{remark}

\begin{proposition} \label{pr:Pushforward1}
  The Gysin map
  $\displaystyle  j_*: \Hoh^0_{\Gm}(X_1
    \cap Y) \to \Hoh^{2(s-1)(r-1)}_{\Gm}(Y)$
  on $\Gm$-equivariant cohomology, induced by the inclusion $(X_1 \cap Y) \hookrightarrow Y$, satisfies $j_*(1)=\theta^{(s-1)(r-1)}$.
\end{proposition}
\begin{proof}
  The inclusion $(X_1 \cap Y) \hookrightarrow Y$ is isomorphic as a map of varieties with $\Gm$-action to the inclusion $\{0\} \to
  \A^{(s-1)(r-1)}_k$, where the affine space is given the standard $\Gm$-action.

  In order to calculate the $\Gm$-equivariant cohomology groups of the spaces concerned, we compare the map with
  \[\{0\} \times (\A^{N+1}_k - \{0 \}) \to \A^{(s-1)(r-1)}_k \times (\A^{N+1} - \{0\})\]
  where $N > (s-1)(r-1)$ and the $\Gm$ action on $\A_k^{N+1}$ is the standard one. Now we take quotients by the $\Gm$-action
  to arrive at 
  \[ z: \P^{N}_k \to \P^{N+(s-1)(r-1)}_k - \P^{(s-1)(r-1)-1}_k,\]
  which is the inclusion of $\P^N_k$ as the zero-section of the bundle $E=\sh O(1)^{(s-1)(r-1)}$.

  The Gysin map $j_*$ above fits in a commutative diagram:
  \[ \xymatrix{ \Hoh^0_\Gm(X_1 \cap Y) \ar^{j_*}[r] \ar^=[d] & \Hoh^{2(s-1)(r-1)}_\Gm(Y) \ar^=[d] \ar^=[dr]\\
      \Hoh^0(\P^N_k) \ar^{z_* \quad \quad}[r] & \Hoh^{2(s-1)(r-1)}(E) \ar_=^{z^*}[r] & \Hoh^{2(s-1)(r-1)}(\P^N_k).} \]
  The arrows marked ``$=$'' represent identifications. 

  The composite $z^* \circ z_*$ sends $1$ to the Euler class of $E$ (Remark \ref{rem:chern}), which is
  $\theta^{(s-1)(r-1)}$ by the multiplicativity of Euler classes \cite[p.~156]{Milnor1974}.
\end{proof}

\begin{lemma} \label{pr:Pushforward2}
  The Gysin map
  \[ \Hoh^*_{\Gm}(X_1 - T_2) \to \Hoh^{* + 2(s-1)(r-1)}_{\Gm}(V_r - T_2) \] on $\Gm$-equivariant
  cohomology, induced by the inclusion $(X_1 - T_2) \hookrightarrow (V_r - T_2)$, sends $1_{X_1 - T_2}$ to $\theta^{(s-1)(r-1)}$. In particular,
  $\Hoh^0_{\Gm}(X_1 - T_2) \to \Hoh^{2(s-1)(r-1)}_{\Gm}(V_r - T_2)$ is nonzero.
\end{lemma}
\begin{proof}
  Consider a square in which the vertical maps are Gysin maps associated to closed embeddings, and in which
  the horizontal maps are induced by the inclusions $Y \cap X_1 \to X_1 \sm T_2$ and $Y \to V_r \sm T_2$:
 \begin{equation*}
   \xymatrix{\ZZ[\theta] \iso \Hoh_{\Gm}^{0}(Y \cap X_1)\ar[d] & \Hoh_{\Gm}^{0}(X_1 - T_2) \iso \ZZ[\theta] \ar[d] \ar[l]
     \\ \ZZ[\theta] \iso \Hoh_{\Gm}^{2(s-1)(r-1)}(Y) &
     \Hoh_{\Gm}^{2(s-1)(r-1)}(V_r - T_2) \ar[l] }
\end{equation*}
This square commutes by Lemma \ref{lem:GysinTransversality}, which applies because $Y$ and $X_1 - T_2$ intersect transversely in
$V_r\sm T_2$ according to Proposition~\ref{prop.XcapY}.

Lemma \ref{pr:Pushforward1} says the left vertical map sends $1_{Y \cap X_1}$ to $\theta^{(s-1)(r-1)}$. The upper
horizontal map is an isomorphism since the varieties in the source and target are connected. Consider the maps induced
by inclusions
  \[
    \begin{tikzcd}
      \Hoh^j_\Gm(V_r) \rar \arrow[bend right]{rr}& \Hoh^j_\Gm(V_r \sm T_2) \rar & \Hoh^j_\Gm(Y).
    \end{tikzcd}
  \]
  Here Remark \ref{not:identifications} tells us that the left arrow is an isomorphism provided $j < 4(s-2)(r-1)+1$ and
  the curved arrow is an isomorphism for all $j$. We deduce that inclusion induces an isomorphism
  $\Hoh^{2(s-1)(r-1)}_\Gm(V_r \sm T_2)\isomto \Hoh^{2(s-1)(r-1)}_\Gm(Y)$.

  The assertion of the lemma is now proved by chasing the diagram.
\end{proof}

\begin{lemma} \label{pr:Pushforward3} The Gysin map
  \[ \Hoh^0_{\PGL_n}(X_1 - T_2) \to \Hoh^{2(s-1)(r-1)}_{\PGL_n}(V_r - T_2) \] on $\PGL_n$-equivariant
  cohomology, induced by the inclusion $(X_1 - T_2) \hookrightarrow (V_r - T_2)$, is nonzero.
\end{lemma}

\begin{proof}
  Lemma \ref{lem:GysinSubgroup} says there is a commutative diagram
 \begin{equation*}
 \xymatrix{\ZZ \iso \Hoh_{\Gm}^{0}(X_1 - T_2)\ar[d] & \Hoh_{\PGL_s}^0(X_1 - T_2) \iso \ZZ \ar[d] \ar[l] \\ \Hoh_{\Gm}^{2(s-1)(r-1)}(V_r - T_2) &
 \Hoh_{\PGL_s}^{2(s-1)(r-1)}(V_r - T_2) \ar[l]}
 \end{equation*}
 in which the upper horizontal arrow takes $1$ to $1$. Lemma \ref{pr:Pushforward2} and a diagram-chase now establishes the result.
\end{proof}

\begin{lemma} \label{pr:MorNotInj}
 The morphism $\Hoh_{\PGL_s}^{2(s-1)(r-1)}( V_r ) \to \Hoh_{\PGL_s}^{2(s-1)(r-1)}(U_r)$ induced by the open immersion $U_r\hookrightarrow V_r$ is not injective.
\end{lemma}

\begin{proof}
 There is a factorization of the inclusion $U_r \hookrightarrow V_r$ as
 \begin{equation*}
 U_r \hookrightarrow V_r - (T_2 \cup X_1) \hookrightarrow V_r - T_2 \hookrightarrow V_r,
 \end{equation*}
 and each of these inclusions is $\PGL_s$-equivariant. By functoriality, it suffices to show
 \[\Hoh_{\PGL_s}^{2(s-1)(r-1)}(V_r )\to \Hoh_{\PGL_s}^{2(s-1)(r-1)}(V_r - (T_2 \cup X_1))\]
 is not injective. Clearly, since $T_2  \subsetneq X_1$, the
 codimension of $T_2$ in $V_r$ is strictly greater than the codimension of $X_1$, i.e., is strictly greater than
 $(r-1)(s-1)$; see Proposition~\ref{prop.Z}(a). By Remark~\ref{not:identifications},
 \[\Hoh_{\PGL_s}^{2(s-1)(r-1)}(V_r) \to \Hoh_{\PGL_s}^{2(s-1)(r-1)}(V_r - T_2) \] is an
 isomorphism. It is therefore sufficient to show that \[ \Hoh_{\PGL_s}^{2(s-1)(r-1)}(V_r - T_2) \to
 \Hoh_{\PGL_s}^{2(s-1)(r-1)}(V_r - (T_2 \cup X_1))\] is not injective. This follows from Proposition~\ref{pr:Pushforward3}
 in combination with the Thom--Gysin sequence 
 \begin{equation*}
 \Hoh^{0}_{\PGL_s}(X_1 - T_2) \to \Hoh^{2(s-1)(r-1)}_{\PGL_s}(V_r - T_2) \to
 \Hoh^{2(s-1)(r-1)}_{\PGL_s}(V_r - (T_2 \cup X_1)). \qedhere
 \end{equation*}
\end{proof}

\begin{remark} \label{rem:prMorNotInjNeeds3} 
(i)  Lemma \ref{pr:MorNotInj} can fail when $s=2$. For instance, if $s=2$ and $r=2$,
  then one can show that $\Hoh_{\PGL_2}^2 (V_2)  \iso \Hoh^2(B\mathrm{PU}_2) = 0$ by
  exploiting the isomorphism of Lie groups between $\SO_3$
  and the projective unitary group $\mathrm{PU}_2$ and
  \cite[Theorem 1.5]{Brown1982}. With some additional effort, we have checked by direct computation that the morphism 
  \[ \Hoh_{\PGL_2}^{2(r-1)}( V_r) \to \Hoh_{\PGL_2}^{2(r-1)}(U_r)\]
  is injective for a few other small even values of $r$. 
  \benw{It's actually a lot of effort}

\smallskip
(ii) The argument in Lemma \ref{pr:MorNotInj} actually shows the slightly stronger result that
  \[ \Hoh_{\PGL_s}^{2(s-1)(r-1)}( V_r ; \QQ ) \to \Hoh_{\PGL_s}^{2(s-1)(r-1)}(U_r; \QQ)\] is not injective, if $s \geqslant 3$.
  By contrast, when $s=2$ and $r$ is even, there is an isomorphism $\Hoh_{\PGL_2}^{2(r-1)}( V_r ; \QQ ) \iso \Hoh^{2(r-1)}(B \SO_3; \QQ)
  = 0$, using \cite[Theorem 1.5]{Brown1982}. 
  
\smallskip
Taken together, (i) and (ii) suggest that the case where $s=2$ and $r$ is even is substantially different from the case where $s \geqslant 3$ considered in~Lemma \ref{pr:MorNotInj}. On the other hand, Gant has shown that this lemma continues to hold when $s=2$ and $r \ge 3$ 
is odd; see \cite[Theorem 5.2]{Gant2020}.
\end{remark}

\begin{proof}[Proof of Theorem \ref{thm.azumaya}(b)]  First, let us suppose the field $k$ is $\QQ$. 
Set $r$ to be the largest integer satisfying $d \geqslant 2(s-1)(r-1)$. In other words, set
\begin{equation} \label{e.pf-azumaya} r = \Floor{ \frac{d}{2(s-1)} } + 1.
\end{equation}
We claim that the natural map $\Hoh^{2(s-1)(r-1)}(B \PGL_s) \to \Hoh^{2(s-1)(r-1)}_{\PGL_s}(U_r)$
  is not injective. This follows from the commutative diagram
  \[ \xymatrix{ & \Hoh^{2(s-1)(r-1)}(B \PGL_s) \ar_{\iso}[dl] \ar[dr] & \\ \Hoh_{\PGL_s}^{2(s-1)(r-1)}( V_r )
      \ar[rr] & &   \Hoh^{2(s-1)(r-1)}_{\PGL_s}(U_r) } \]
  along with Lemma \ref{pr:MorNotInj}.

  We now apply Theorem \ref{thm:produceCounterexample} with $i = 2(s-1)(r-1)$ 
  to produce a $\QQ$-ring $R'$ and an Azumaya $R'$-algebra $B'$ such that $\Kdim R' = i = 2(s-1)(r-1)$ and
  $\gen_{R'_{\CC}}(B'_{\CC}) \geqslant r + 1$. Recall that by our choice of $r$, we have $i \leqslant d$. After replacing $R'$ by the polynomial ring
  $R = R'[x_1, \ldots, x_{d-i}]$ and $B'$ by $B = B' \otimes_{R'} R$,
  we arrive an an Azumaya algebra $B$ over $R$ such that $\Kdim R = d$ and $\gen_{R_{\CC}}(B_{\CC}) = \gen_{R'_{\CC}}(B'_{\CC}) \ge r + 1$. 
  In particular, $\gen_{R}(B) \geqslant r + 1$; see Lemma~\ref{lem:requireForLefschetz}. This completes
  the proof of Theorem~\ref{thm.azumaya}(b) in the case, where $k = \QQ$.
  
  In the general case, where $k$ is an arbitrary field of characteristic $0$ and not necessarily embedded in $\CC$, we invoke the Lefschetz principle of 
  Lemma~\ref{lem:requireForLefschetz}. Consider the ring $R$
  and Azumaya algebra $B$ produced in the case of $\QQ$. Set $R_k = R \otimes_{\QQ} k$ and $B_k = R \otimes_{\QQ} k$.
  Clearly $\Kdim R_k = \Kdim R = d$. Moreover,
  \[ \gen_{R_k}(B_k) \ge \gen_{R_\CC}(B_\CC) \ge r+1 =\Floor{ \frac{d}{2(s-1)} } + 2  \] 
  by Lemma \ref{lem:requireForLefschetz}(c). 
\end{proof}

\appendix
\section{The equivariant Thom--Gysin sequence for varieties}

\label{sect.topological-preliminaries}

This appendix proves a number of results about the Borel-equivariant singular cohomology of a $k$-variety,
$\Hoh^*_G(X, \ZZ)$, when $k \subset \CC$. These are used in Sections \ref{s:equivCohoUr}, \ref{sect.main3}
and \ref{sect.azumaya(b)}. Many of them are well known, and we observe in passing some of the places in the literature
where they are used, but clear statements and proofs are not easy to find.

We have made use of the following technical lemma in a few places in the text, and will use it again in this appendix.
\begin{lemma} \label{lem:localizationIso}
  Let $X$ be a smooth variety. If $Z \to X$ is a possibly-nonsmooth closed subvariety of codimension
  $d$---that is, if each irreducible component of $Z$ is of codimension at least $d$ in $X$---and if $U = X \sm Z$, then
  the pullback map $\Hoh^i(X) \to \Hoh^i(U)$ is an isomorphism when $i<2d-1$ and is injective when $i=2d-1$.
\end{lemma}

This is standard using Borel--Moore homology. Specifically, the exact sequence \cite[IX.2 (2.1)]{Iversen1986}
relates $\Hoh^{\rm BM}_\ast(X)$, $\Hoh^{\rm BM}_\ast(U)$ and $\Hoh^{\rm BM}_\ast(Z)$; the groups $\Hoh^{\rm BM}_i(Z)$ vanish unless $0
\le i \le \dim Z$; and there are isomorphisms $\Hoh_i^{BM}(X) \iso \Hoh^{\dim X - i}(X)$ and similarly for $U$.

\subsection*{The Thom--Gysin sequence}

If $i \colon Z \to X$ is the inclusion of a closed smooth subvariety of constant codimension $c$ and if
$j \colon X \sm Z \to X$ is the inclusion of the complement, then there is a long exact \textit{Thom--Gysin sequence}
\begin{equation}\label{eq:thgy}
  \xymatrix{ \cdots \ar[r] & \Hoh^{*}(X\sm Z) \ar@{->}[r]^{\bd} & \Hoh^{*-2c}(Z) \ar^{i_*}[r] & \Hoh^*(X) \ar@{->}[r]^{j^* \; \; } & \Hoh^*(X \sm Z) \ar@{->}[r]^{\; \; \bd}
    & \cdots. } 
\end{equation}
This sequence is widely known, appearing in ~\cite{Gordon1975}, \cite[p.~167]{Mumford1994} for instance. The special case where $X$ is the
total space of a vector bundle and $Z$ the zero section is the Gysin sequence of \cite[Theorem
12.2]{Milnor1974}. In the general case, one can find tubular neighbourhood $\nu$ of $Z$ in $X$, and deduce the
Thom--Gysin sequence for $Z$ in $X$ from the case of $Z$ in $\nu$ using the excision theorem.
Alternatively, one may derive the sequence \eqref{eq:thgy} from the long exact sequence of local cohomology
\cite[Prop.~II.9.2]{Iversen1986} along with the isomorphism $\Hoh^{\ast}_Z(X; \ZZ) \to \Hoh^{\ast-2c}(Z; \ZZ)$ implied
by \cite[Theorem IX.4.7]{Iversen1986}.

In order to articulate the naturality of the Thom--Gysin sequence, we recall the concept of a transversal square. Let
  \begin{equation} \label{diag:transversal}
    \xymatrix{  Z' \ar@{->}[r]^{i'} \ar@{->}[d]_f &  X' \ar@{->}[d]^g \\ Z  \ar@{->}[r]^i & X}
  \end{equation}
 be a cartesian square of smooth varieties, where $i$ is a closed embedding. Let $N'$ denote the normal bundle of
 $Z'$ in $X'$ and $N$ the normal bundle of $Z$ in $X$. Recall from \cite[\S
 17.13]{Grothendieck1967} that the square is \textit{transversal} if the induced map
 $N' \to f^* N$ is an isomorphism.

The most important case is when $X' \to X$ and $Z \to X$ are both closed
embeddings. In this case, $X'$ and $Z$ are said to \textit{intersect transversely}, and $Z'$ is
their transverse intersection.

Intersecting transversely is a condition on normal bundles. Using the normal exact sequence
\cite[p.~182]{HartshorneAlgebraicGeometry1977}, we can say the intersection is transverse if
\[ \dim T_p X = \dim T_p X' + \dim T_p Z - \dim T_p Z' \]
for all closed points $p$ of the smooth variety $Z'$.

\begin{remark} \label{rem:transverseFFP} A square is transversal if a certain map of locally free sheaves is an
  isomorphism. This holds for the square defined by $Z\subseteq X$ and $X' \to X$ if and only if it holds for the square
  obtained by pulling back along a faithfully flat map $U \to X$.
\end{remark}

If we are given a transversal square, then there is an induced natural transformation of Thom--Gysin sequences.
\begin{equation} \label{eq:transversalLadder}
  \xymatrix{ \cdots \ar[r] & \Hoh^{*}(X\sm Z) \ar[d] \ar^\bd[r] & \Hoh^{*-2c}(Z) \ar^{i_*}[r] \ar^{f^*}[d]  & \Hoh^*(X) \ar^{j^*}[r]  \ar^{g^*}[d] & \Hoh^*(X
    \sm Z) \ar^{\bd}[r] \ar[d] 
    & \cdots \\ \cdots
    \ar[r] & \Hoh^*( X' \sm Z') \ar^\bd[r] & \Hoh^{* - 2c}(Z') \ar^{i'_*}[r] & \Hoh^*(X') \ar[r] & \Hoh^*(X' \sm Z') \ar[r] & \cdots} 
\end{equation}
This and other useful properties of Thom--Gysin sequence are listed in \cite[Section 2.2]{Panin2009}. 

\begin{remark} \label{rem:chern}
  If $s: X \to V$ is the zero-section of a rank-$m$ vector bundle on a smooth variety, then the composite of the Gysin map $s_*$ with $s^*$
  \[\Hoh^0(X) \overset{s_*}\to \Hoh^{2m}(V) \overset{s^*}\to \Hoh^{2m}(X)\]
  sends $1_X$ to the Euler class $e(V)$. We sketch the argument: In the case of a zero-section of a bundle, the definition in \cite[Section 2.2]{Panin2009} of the
  Gysin map coincides with a composite of taking a cup-product with the Thom class of the bundle, followed by extension
  of support:
  \[ \xymatrix{ \Hoh^{*-2m}(X) \ar^{\smile \tau}[r] & \Hoh^*_X(V) \ar[r] & \Hoh^*(V)}. \]
  Composing further with $s^*$ gives the definition of Euler class in \cite[Definition 12.1]{Bredon1993}.

  It is also the case that $\Hoh^*(\PP^n) \iso \ZZ[\theta]/(\theta^{n+1})$, where $\theta$ is the first
  Chern class of $\sh O(1)$, so that $|\theta| = 2$.
\end{remark}

\subsection*{Equivariant Thom--Gysin sequence}

  There exists an equivariant Thom--Gysin sequence. Versions of this appear in \cite[p.~170]{Kirwan1984} and
  \cite{Crooks2016}, for instance. In its Borel--Moore homology incarnation, it is implied by the remarks
  on \cite[Sec.~2.8]{EdidinEquivariantintersectiontheory1998}. In fact, this appendix consists of an explication of some of what is implicit in \cite[Sec.~2.8]{EdidinEquivariantintersectiontheory1998}.

  Let $G$ be an algebraic group.
  If $Z \to X$ is a closed $G$-invariant embedding of constant codimension $c$ where $Z$ and $X$ are both smooth varieties, then
  in any range of dimensions we may find some representation $V$ of $G$ and open subspace $U \subset V$ in order to
  calculate both $\Hoh^i_G(Z) = \Hoh^i(Z \times^G U)$ and $\Hoh^*_G(X) =\Hoh^i(X \times^G U)$. The inclusion
  $(Z \times U)/G \to (X \times U)/G$ is a closed embedding of smooth varieties, again of codimension $c$. This means
  there is a Gysin map
  \[ \Hoh^{*-2c}_G(Z) \to \Hoh^*_G(X) \]
  exactly as in the non-equivariant case. 

\begin{lemma} \label{lem:transverseEquivariant} Let $G$ be an algebraic group acting freely on a smooth variety $X$ in
  such a way that $X \to X/G$ is a $G$-torsor in the category of smooth varieties. Suppose $Y$ and $Z$ are $G$-invariant
  closed smooth subvarieties of $X$ that intersect transversely. Then $Y/G$ and $Z/G$ intersect transversely in $X/G$.
\end{lemma}
\begin{proof}
  By using the observations of Remark \ref{rem:transverseFFP}, we can check for the transversality of $Y/G$ and $Z/G$
  after faithfully flat pullback to $X$, but these pullbacks are $Y$ and $Z$, which intersect transversely.
\end{proof}

\begin{lemma} \label{lem:GysinTransversality}
  Let $G$ be an algebraic group acting on a smooth variety $X$. Suppose $Y$ and $Z$ are smooth closed $G$-subvarieties
  intersecting transversely, and $Y$ is of constant codimension $c$. Then there is a natural transformation of equivariant Thom--Gysin sequences
  \begin{equation*}
  \xymatrix{  \ar[r] & \Hoh_G^{*}(X\sm Y) \ar[d] \ar^\bd[r] & \Hoh_G^{*-2c}(Y) \ar[r] \ar[d]  & \Hoh_G^*(X) \ar^-{j^*}[r]  \ar[d] & \Hoh_G^*(X
    \sm Y) \ar^-{\bd}[r] \ar[d] 
    &  \\ 
  \ar[r] & \Hoh_G^*(Z \sm (Y \cap Z)) \ar^\bd[r] & \Hoh_G^{*-2c}(Y \cap Z) \ar[r] & \Hoh_G^*(Z) \ar[r] & \Hoh_G^*(Z \sm (Y \cap Z)) \ar[r] & } 
\end{equation*}
\end{lemma}
\begin{proof}
  The groups $\Hoh_G^i$ may be calculated by choosing a $G$-representation $V$ and an open subset $U$ where $U \to U/G$ is a
  $G$-torsor and $V \sm U$ has sufficiently high codimension. Then $\Hoh^i_G(X)$  is the cohomology of $(X \times U)/G$ and
  similarly for the other varieties. It
  suffices therefore to prove that $(Y \times U)/G$ and $(Z \times U)/G$ intersect transversely in $(X \times U)/G$, but
  this is implied by Lemma \ref{lem:transverseEquivariant}.
\end{proof}

\begin{lemma} \label{lem:GysinSubgroup}
  Let $G$ be an algebraic group and $H$ a subgroup. Suppose $Z$ is a $G$-invariant closed smooth subvariety of a smooth $G$-variety
  $X$ of constant codimension $c$. Then there is a natural transformation of equivariant Thom--Gysin sequences:
   \begin{equation*}
  \xymatrix{ \cdots \ar[r] & \Hoh_G^{*}(X\sm Z) \ar[d] \ar^\bd[r] & \Hoh_G^{*-2c}(Z) \ar[r] \ar[d]  & \Hoh_G^*(X) \ar^-{j^*}[r]  \ar[d] & \Hoh_G^*(X
    \sm Z) \ar^-{\bd}[r] \ar[d] 
    & \cdots \\ \cdots
  \ar[r] & \Hoh_H^*(X \sm Z) \ar^\bd[r] & \Hoh_H^{*-2c}(Z) \ar[r] & \Hoh_H^*(X) \ar[r]^-{j^*} & \Hoh_H^*(X-Z) \ar[r]^-{\partial} & \cdots} 
\end{equation*}
\end{lemma}
\begin{proof}
  Letting $U$ be as in the proof of
  Lemma~\ref{lem:GysinTransversality},
  it suffices to prove that
  \begin{equation*}
    \xymatrix{ (Z \times U)/H \ar[r] \ar[d] & (X\times U)/H \ar[d] \\ (Z \times U)/G \ar[r] & (X \times U)/G }
  \end{equation*}
  is transversal. This can be verified after faithfully flat base change along $(X \times U) \to (X \times U)/G$,
  whereupon it is isomorphic to 
\begin{equation*}
    \xymatrix{ Z \times U \times (G/H) \ar[r] \ar[d] & X\times U \times (G/H) \ar[d] \\ Z \times U \ar[r] & X \times U. }
  \end{equation*}
  The maps are either projections or closed embeddings. This is obviously transversal.
\end{proof}

\begin{lemma} \label{lem:localizationIsoEquiv}
  Let $X$ be a smooth variety equipped with a $G$-action. If $Z \to X$ is a possibly-nonsmooth closed $G$-subvariety of
  codimension $d$ and if $U = X - Z$, then the pullback map $\Hoh^i_G(X) \to \Hoh^i_G(U)$ is an isomorphism when
  $i < 2d-1$ and is injective when $i=2d-1$.  
\end{lemma}
\begin{proof}
  Choose a $G$-representation $V$ and an open subset $U$ such that $U \to U/G$ is a $G$-torsor and $V\sm U \into V$ has
  codimension greater than $d$. The result now follows by applying Lemma \ref{lem:localizationIso} to $(Z \times U)/ G
  \to (X \times U)/ G$.
\end{proof}

{\setcounter{biburlnumpenalty}{100} \printbibliography}
\end{document}
